\documentclass[11pt]{amsart}
\usepackage[a4paper,margin=2.5cm]{geometry}

\usepackage[T1]{fontenc}
\usepackage{lmodern}
\usepackage{microtype}
\usepackage{graphicx}
\usepackage[normalem]{ulem}
\usepackage{cancel}

\usepackage{amsmath,amssymb,amsthm,mathtools}
\mathtoolsset{showonlyrefs}
\usepackage{braket}
\usepackage{aliascnt}

\usepackage[hypertexnames=false]{hyperref}

\newcommand{\R}{\mathbb{R}}

\newcommand{\cK}{\mathcal{K}}

\DeclarePairedDelimiter{\abs}{\lvert}{\rvert}
\DeclarePairedDelimiter{\paren}{(}{)}

\newcommand{\ch}{\mathfrak{ch}}
\newcommand{\pa}{\mathfrak{pa}}
\newcommand{\de}{\mathfrak{de}}
\newcommand{\debar}{\overline{\de}}

\newcommand{\1}[1]{\mathbf{1}_{#1}}

\newcommand{\GB}{\mathrm{GB}}
\newcommand{\TB}{\mathrm{TB}}

\newcommand{\Beta}{\mathrm{Beta}}

\newcommand{\GBI}{\GB_{\mathrm{I}}}
\newcommand{\BetaI}{\Beta_{\mathrm{I}}}
\newcommand{\GBII}{\GB_{\mathrm{II}}}
\newcommand{\BetaII}{\Beta_{\mathrm{II}}}

\numberwithin{equation}{section}

\theoremstyle{plain}
\newtheorem{theorem}{Theorem}[section]

\newaliascnt{lemma}{theorem}
\newtheorem{lemma}[lemma]{Lemma}
\aliascntresetthe{lemma}

\newaliascnt{proposition}{theorem}
\newtheorem{proposition}[proposition]{Proposition}
\aliascntresetthe{proposition}

\newaliascnt{corollary}{theorem}
\newtheorem{corollary}[corollary]{Corollary}
\aliascntresetthe{corollary}

\theoremstyle{definition}
\newaliascnt{definition}{theorem}
\newtheorem{definition}[definition]{Definition}
\aliascntresetthe{definition}

\newaliascnt{example}{theorem}

\aliascntresetthe{example}

\theoremstyle{remark}
\newtheorem{remark}{Remark}

\title[Independence properties for tree beta models]{Independence properties for tree beta models}
\author{Yoshihiro Gyotoku}
\address{Graduate School of Mathematical Sciences, The University of Tokyo,
3-8-1 Komaba, Meguro-ku, Tokyo 153-8914, Japan}
\email{gyotoku@ms.u-tokyo.ac.jp}

\author{Makiko Sasada}
\address{Graduate School of Mathematical Sciences, The University of Tokyo,
3-8-1 Komaba, Meguro-ku, Tokyo 153-8914, Japan}
\email{sasada@ms.u-tokyo.ac.jp}

\thanks{
This work was supported by JSPS KAKENHI Grant Nos.~24KJ0933 and 24K21515, by the FoPM WINGS Program at the University of Tokyo, and  by the National Science Centre, Poland, under Project
No.~2023/51/B/ST1/01535.
}

\begin{document}

\author{Jacek Weso\l owski}
\address{Faculty of Mathematics and Information Science, Warsaw University of Technology, Koszykowa 75, 00-662 Warsaw, Poland}
\email{jacek.wesolowski@pw.edu.pl}
\begin{abstract} 
For an undirected tree $G = (V, E)$ with edge weights $q = (q_{e})_{e \in E}$ we consider tree polynomials
	\begin{equation}
		\Delta_{G}(k) \coloneqq \sum_{U \subseteq V} (-1)^{|U|}q^{E_U}k^{U}\quad \text{and}\quad \delta_G(k)\coloneqq \sum_{U \subseteq V}\,q^{E_U}k^{U}
	\end{equation}
	where $q^{E_U} \coloneqq \prod_{e \in E_U} q_{e}$ and $k^{U} \coloneqq \prod_{v \in U} k_{v}$. 
	We introduce  multivariate first/second kind beta-type probability distributions,  $\TB^{G}_{\mathrm I}/\TB^{G}_{\mathrm{II}}$, whose densities with respect to Lebesgue measure are proportional to
	\begin{equation}
		\Delta_{G}(k)^{\beta-1} \prod_{v \in V} \abs{k_{v}}^{a_{v} - 1} \1{\cK}(k) \quad\text{and}\quad \delta_{G}(k)^{- \beta} \prod_{v \in V} \abs{k_{v}}^{a_{v} - 1} \1{\cK}(k)
	\end{equation}
    on suitable supports $\cK\subset \mathbb R^V$ and with suitable range of parameters $a=(a_v) \in\mathbb R^V$, $\beta\in\mathbb R$ and $q=(q_e)\in\mathbb R^E$. 
	For each vertex $r \in V$, the rooted directed tree $G^{(r)}$ determines a mapping $\Psi^{(r)}/\psi^{(r)}$ whose components are ratios of tree polynomials $\Delta_G/\delta_G$ evaluated on descendant induced subgraphs of $G$.
	It is shown that if a random vector $K \sim \TB^{G}_{\mathrm I}/\TB^{G}_{\mathrm{II}}$, then $\Psi^{(r)}(K)/\psi^{(r)}(K)$ is a vector of independent generalised $\Beta_{\mathrm I}/\Beta_{\mathrm {II}}$ random variables for every root~$r$. Conversely, if a non-Dirac random vector $K$ is such that $\Psi^{(\ell)}(K)/\psi^{(\ell)}(K)$ has independent non-Dirac components for every leaf $\ell$ of $G$, then $K$ follows a $\TB^{G}_{\mathrm I}/\TB^{G}_{\mathrm{II}}$ law.

	Closed-form normalizing constants for both tree-beta distributions, expressed as products of Euler beta/gamma functions and Gauss hypergeometric functions, are also derived. Finally, it is proved that the second-kind tree beta law converges under a natural scaling to the tree-Kummer law. Recently, \cite{SasUoz2024} identified probabilistic models with the independence preserving (IP) property related to a hierarchy of quadrirational Yang-Baxter layers of maps $H_{\mathrm{I}}^+$, $H_{\mathrm{II}}^+$ and $H_{\mathrm{III}}$. When $\abs{E} = 1$ our models are closely related to the models of the $H_{I}^{+}$ layer. The tree beta models we introduce parallel the tree-GIG results of \cite{MasWes2004} (related to the models with the IP properties of the $H_{\mathrm{III}}$ layer) and the tree-Kummer results of \cite{PilWes16} (related to the models with the IP properties of the $H_{II}^{+}$ layer). When $q \equiv 0$, our models reduce to  tree versions of the Dirichlet and inverse Dirichlet models on decomposable graphs of \cite{DKWZ25}. 

\end{abstract}

\maketitle

\section{Introduction and Main Results} \label{sec:introduction}



A map $F:\mathcal{X}\times\mathcal{Y}\to\mathcal{U}\times\mathcal{V}$ is called independence preserving (IP) if there exist non-Dirac probability distributions $\mu$, $\nu$, $\tilde{\mu}$, and $\tilde{\nu}$ on the respective spaces such that if $X\sim \mu$ and $Y\sim \nu$ are independent, then the variables $(U,V)=F(X,Y)$ satisfying $U\sim\tilde{\mu}$ and $V\sim\tilde{\nu}$ are also independent. 
The identification of the entire family of probability measures $\mu,\nu,\tilde{\mu},\tilde{\nu}$ for a given map $F$ is called the independence characterization problem induced by $F$. 

Classical examples include the Kac--Bernstein characterization of the Gaussian law via the independence of the sum and difference, and the Lukacs characterization of the gamma law via the independence of the sum and ratio. 

Recently, \cite{SasUoz2024} identified a hierarchy of IP maps within the classification of quadrirational Yang--Baxter transformations discussed by \cite{PSTV2010}. Specifically, maps from the $H_{\mathrm{III}}$ layer of this classification were identified as IP maps for the generalized inverse Gaussian (GIG) distributions, and maps from the $H_{\mathrm{II}}^+$ class were identified as IP maps for the Kummer distribution. 
Furthermore, maps from the $H_{\mathrm{I}}^+$ class, which stand at the top of the hierarchy, were identified as IP maps for the generalized beta prime (also called $\mathrm{Beta}_{\mathrm{II}}$) distribution. Respective characterizations for all these cases were established, respectively, in \cite{LetWes2024}, \cite{KouWes2025}, and \cite{DKWZ25}. 

All this refers to the bivariate setting. To introduce the multivariate setting we first remark that in the bivariate case alternatively one can consider a random vector $K=(Y,V)$. Then the IP map $F$ generates two maps, say, $\chi_1$ and $\chi_2$ such that random vectors $\chi_1(K)=(X,Y)$ and $\chi_2(K)=(U,V)$ have independent components. Such approach is convenient to lift the bivariate IP models to multivariate ones. Thus, in the multivariate setting, we  search for a random vector $K$ (or strictly speaking, its distribution) valued in some set $\mathcal{K}\subset \mathbb{R}^d$, alongside a family of bijections $\Psi^{(r)}:\mathcal K\to \mathcal X$ for some domains $\mathcal K,\mathcal X\subset \mathbb R^d$, $r\in\{1,\ldots,d\}$, such that each of the $d$-variate random vectors $\Psi^{(r)}(K)$ has independent components. We will call such functions independence generating (IG) maps. Consequently,  compositions of an IG map with the inverse of another IG map, e.g. $\Psi^{(s)}\circ(\Psi^{(r)})^{-1}$ for distinct $r,s\in\{1,\ldots,d\}$, are multivariate IP maps, i.e. each such composition transforms a random vector with independent components into another random vector with independent components. In this paper we are interested in tree-generated models in which both the distribution of the random vector $K$ has a tree-structure and maps $\Psi^{(r)}$, $r=1,\dots,d$, are related to respective directed trees. Two such tree models are well-documented in the literature and are related to bivariate $H_{\mathrm{III}}$ and $H_{\mathrm{II}}^+$ families of Yang-Baxter maps mentioned above. To briefly describe them we need to introduce some graph theoretical basics. 

Let $G=(V,E)$ be a tree with a vertex set $V$ and an edge set $E$. The standing assumption in this paper is that $1 < |V| <\infty$.  For $U\subset V$, we denote by $G_U=(U,E_U)$ the induced subgraph of $G$ (i.e., $E_U$ is the set of edges $e=\{v,w\}\in E$ such that $v,w\in U$). We denote by $L_G$ the set of its leaves, that is, vertices $v\in V$ of degree 1.
We are interested in the setting where, for each $r\in V$, the IG map $\Psi^{(r)}=(\Psi^{(r)}_v,\,v\in V)$ is constructed with reference to the $r$-rooted directed tree $G^{(r)}$ sharing the common skeleton $G$. To describe the structure of $G^{(r)}$, we write $v\to w$ if $\{v,w\}\in E$ and $v$ lies on the unique path connecting $r$ and $w$. Accordingly, for each $v\in V$, we denote the set of children of $v$ in the directed tree $G^{(r)}$ by $\ch^{(r)}(v):=\{w\in V:\,v\to w\}$. 
    
Additionally, we introduce the edge weights $(q_e)_{e\in E}$, where we write $q_{v,w}:=q_e$ when $e=\{v,w\}\in E$. 

\bigskip
The first model is the Matsumoto--Yor (MY) tree model (see \cite{MasWes2004}). It can be considered as a multivariate version of certain bivariate IP models from the $H_{\mathrm{III}}$ class (see \cite{SasUoz2024}), which originate from \cite{MatYor2001}, \cite{MatYor2003}. The MY tree model has the following two basic ingredients: 

\begin{enumerate}
    \item The tree-GIG distribution defined by the density
\begin{equation}\label{tGIG}
f(k)\propto \Big[\det (Q+\mathrm{diag}(k)) \Big]^{\beta-1}\,e^{-\sum_{v\in V}\,a_vk_v}\,\mathbf{1}_{\mathcal{K}}(k)
\end{equation}
with parameters $\beta>0$ and $a=(a_v)\in(0,\infty)^V$. 
Here $Q$ is the $(q_e)_{e\in E}$-weighted incidence matrix of $G$ with $q_e\neq 0$, $e\in E$, $\mathrm{diag}(k)$ is the diagonal matrix with the vector $k$ on the diagonal, and  
\[
\mathcal{K} = \{k=(k_v)\in(0,\infty)^V\mid Q+\mathrm{diag}(k)>0\}.
\]
\item The maps $\Psi^{(r)}:\mathcal K\to(0,\infty)^V$, $r\in V$, defined by
\begin{equation}\label{PsiMY}
\Psi^{(r)}_v(k)=k_v-\sum_{w\in\ch^{(r)}(v)}\,\frac{q_{w,v}^2
}{\Psi^{(r)}_w(k)},\quad k=(k_v)_{v\in V}\in\mathcal K,\quad v\in V.
\end{equation}
(This is a recursive definition which starts with $v\in L_G\setminus\{r\}$; then $\ch^{(r)}(v)=\varnothing$). 
\end{enumerate} 
To see that the image is $(0,\infty)^V$ see, e.g., Lemma 2.3 of \cite{MasWes2004}.
It appears that $\Psi^{(r)}$, $r\in V$, are IG maps for the tree-GIG distribution: If a random vector $K$ has the tree-GIG distribution, then for any $r\in V$, the components of the random vector $\Psi^{(r)}(K)=:X^{(r)}=(X^{(r)}_v)_{v\in V}$ are independent, and $X^{(r)}_v$, $v\in V\setminus\{r\}$ have GIG (generalized inverse Gaussian) distributions, while $X^{(r)}_r$ is a gamma random variable. Conversely, if $K$ is a non-Dirac random vector with distribution supported on $\mathcal K$, then the independence of the components of $\Psi^{(r)}(K)$ just for all $r\in L_G$ implies that $K$ has a tree-GIG distribution with density \eqref{tGIG}. Interestingly, the MY tree model has a matrix variate counterpart, \cite{Bob2015}.

\bigskip
The second model is the Hamza--Vallois (HV) tree model (see \cite{PilWes16}). It can be considered as a multivariate version of the special bivariate IP models from the $H_{\mathrm{II}}^+$ class (see \cite{SasUoz2024}), which originate from \cite{HamVal2016}. The HV tree model has the following two basic ingredients:
\begin{enumerate}
    \item The tree-Kummer distribution, $\mathrm{TK}^G(a,\beta,q)$, for $a\in (0,\infty)^V$, $\beta>0$ and $q\in(0,\infty)^E$, is defined by the density
\begin{equation}\label{tKum}
f(k)\propto \exp\Biggl(-\beta\sum_{\substack{U \subseteq V \\G_U \, \text{is a tree}}} q^{E_U}k^{U} \Biggr)\;\Biggl(\prod_{v\in V}\,k_v^{a_v-1}\Biggr)\,\mathbf{1}_{(0,\infty)^V}(k),
\end{equation}
where $q^{E_U}=\prod_{e\in E_U}\,q_e$ and $k^{U}=\prod_{v\in U}\,k_v$.
\item The maps $\Psi^{(r)}:(0,\infty)^V\to(0,\infty)^V$ defined by
\begin{equation}\label{PsiHV}
\Psi^{(r)}_v(k)=k_v\prod_{w\in\ch^{(r)}(v)}\,\left(1+q_{w,v}\,\Psi^{(r)}_w(k)\right),\quad k=(k_v)_{v\in V}\in(0,\infty)^V, \quad v\in V.
\end{equation}
(This is a recursive definition which starts with $v\in L_G\setminus \{r\}$; then $\ch^{(r)}(v)=\varnothing$).
\end{enumerate}
It appears that $\Psi^{(r)}$, $r\in V$, are IG maps for the tree-Kummer distribution: If a random vector $K$ has the tree-Kummer distribution, then for any $r\in V$, the components of the random vector $\Psi^{(r)}(K)=:X^{(r)}=(X^{(r)}_v)_{v\in V}$ are independent, and $X_v^{(r)}$, $v\in V\setminus \{r\}$ have Kummer distributions, while $X^{(r)}_r$ is a gamma random variable. Conversely, if $K$ is a non-Dirac random vector valued in $(0,\infty)^V$, then the independence of the components of $\Psi^{(r)}(K)$ just for all $r\in L_G$ implies that $K$ has a tree-Kummer distribution with density \eqref{tKum}. (Actually, the tree-Kummer distribution described above is the distribution defined in Sec. 3.3. \cite{PilWes16} with slightly simplified parameters.)

\bigskip
The main goal of this paper is to identify and analyze independence properties of the missing $H_{\mathrm{I}}^+$-related tree model. 
Actually we will identify two such models---featuring the tree-generalized-$\mathrm{Beta}_{\mathrm{I}}$ and the tree-generalized-$\mathrm{Beta}_{\mathrm{II}}$ distributions and their related IG maps parametrized by the edge weights $(q_e)_{e\in E}$ and directed graphs $G^{(r)}$, $r\in V$. 

Notably, as we will see, in the case where $q_e=0$ for all $e\in E$, these tree beta-type models  reduce to the tree restriction of the Dirichlet-type models on decomposable graphs, which were recently introduced and studied by \cite{DKWZ25} in the context of parametric Bayesian graphical models of negative multinomial and multinomial type. 
In particular, the authors of \cite{DKWZ25} considered random vectors $K$ with two Dirichlet-type graph-related laws for a decomposable graph $G=(V,E)$,  which we briefly describe below. To do it, we need to recall some basic facts about decomposable graphs, see e.g. \cite{BlaPey1993}.

Let $G=(V,E)$ be a finite undirected simple graph.  The complementary graph of $G$ is the graph $G^*=(V,E^*)$, where $E^*=\{\{v,w\}:\,v\neq w,\,v,w\in V,\;\{v,w\}\not\in E\}$. A subset $C\subset V$ is a clique in $G$ if $G_C$ is a complete graph. By $\mathcal C_G$ we denote the set of cliques of $G$. In particular, $\varnothing\in\mathcal C_G$. Consider a directed graph $\mathcal G$ with skeleton $G=(V,E)$. Denote by $\mathfrak{pa}^{(\mathcal G)}(v)=\{w\in V:\,w\to v\}$ the set of parents of $v\in V$. We say that $\mathcal G$ is a DAG (directed acyclic graph) if it has no directed cycles. A DAG $\mathcal G$ is moral if for any $v\in V$ the induced graph $G_{\mathfrak{pa}^{(\mathcal G)}(v)}$ is complete. One of many equivalent characterizations of decomposability says that a graph $G=(V,E)$ is decomposable iff there exists a moral DAG $\mathcal G$ with skeleton $G$. In \cite{DKWZ25} the following two models were introduced and investigated, in particular in the context of application to Bayesian graphical models. 
\begin{enumerate}
    \item {\em Graph Dirichlet distribution}: Let $K$ be a random vector with distribution defined by the density
    \[
    f(k)\propto[\Delta_G(k)]^{\beta-1}\,\biggl(\prod_{v\in V}\,k_v^{a_v-1}\biggr)\,\mathbf 1_{M_G}(k),
    \]
    where  $a\in(0,\infty)^V$, $\beta>0$, 
    \begin{equation}\label{dDelta}
    \Delta_G(k)=\sum_{C\in \mathcal C_{G^*}}\,(-1)^{|C|}\,k^C
    \end{equation}
     and 
    \[
    M_G=\{k\in (0,1)^V :\,\Delta_A(k)>0\, \ \forall A\subset V\}
    \] 
    with $\Delta_A(k):=\Delta_{G_A}(k|_A)$.
     For any moral DAG $\mathcal G$ with skeleton $G$, define a map
    \[
    \Psi^{(\mathcal G)}_v(k)=\tfrac{k_v}{\prod_{w\in\mathfrak{ch}^{(\mathcal G)}(v)}\,(1-\Psi^{(\mathcal G)}_w(k))},\quad  v\in V,\quad \text{for }k\in M_G.
    \]
    Then $\Psi^{(\mathcal G)}$, where $\mathcal G$ is a moral with decomposable skeleton $G$, are IG maps in this model: random vector $\Psi^{(\mathcal G)}(K)=:X^{(\mathcal G)}=(X_v^{(\mathcal G)})_{v\in V}$ has independent $\mathrm{Beta}_{\mathrm{I}}$ components. 
    
    Conversely, if $K$ is a non-Dirac $M_G$-valued random vector, then the independence of components of $\Psi^{(\mathcal G)}(K)$ for all moral DAGs $\mathcal G$ with skeleton $G$ implies that $K$ has the graph-Dirichlet distribution with the density as given above. 
    
    Let us point out that when $G$ is a complete graph, the graph Dirichlet law reduces to the classical Dirichlet distribution. In this case, the independence properties described above correspond to the concepts known as neutrality, complete neutrality and, more generally, neutrality with respect to partitions. Neutrality related characterizations of the Dirichlet distribution have been intensively studied in the literature, see e.g. \cite{Fabius1973}, \cite{JamMos1980}, \cite{GeiHec1997}, \cite{BobWes2007}, \cite{BobWes2009} and \cite{MasWes2016}.
    
    \item {\em Graph inverted Dirichlet distribution:} Let $K$ be a random vector with distribution defined by the density
    \[
    f(k)\propto[\delta_G(k)]^{-\beta}\,\biggl(\prod_{v\in V}\,k_v^{a_v-1}\biggr)\,\mathbf 1_{(0,\infty)^V}(k),
    \]
    where $a\in (0,\infty)^V$ and $\beta>\max_{C\in\mathcal C_G}\,\biggl(\sum_{v\in C}\,a_v\biggr)$ and
    \begin{equation}\label{ddelta}
    \delta_G(k)=\sum_{C\in \mathcal C_{G^*}}\,k^C.
    \end{equation}
    For any moral DAG $\mathcal G$ with decomposable skeleton $G=(V,E)$, define a map $\psi^{(\mathcal G)}=(\psi^{(\mathcal G)}_v)_{v\in V}$ by
    \[
    \psi_v^{(\mathcal G)}(k)=\tfrac{k_v}{\prod_{w\in\mathfrak{ch}^{(\mathcal G)}(v)}\,(1+\psi_w^{(\mathcal G)}(k))}, \quad v\in V,\quad \text{for } k \in(0,\infty)^V.
    \]
    Then $\psi^{(\mathcal G)}$, where $\mathcal G$ is a moral with decomposable skeleton $G$, are IG maps in this model: the random vector $\psi^{(\mathcal G)}(K)=:X^{(\mathcal G)}=(X_v^{(\mathcal G)})_{v\in V}$ has independent $\mathrm{Beta}_{\mathrm{II}}$ components.

    Conversely, if $K$ is a non-Dirac $(0,\infty)^V$-valued random vector, then the independence of components of $\psi^{(\mathcal G)}(K)$ for all moral DAGs $\mathcal G$ with skeleton $G$ implies that $K$ has the graph inverted Dirichlet distribution with the density as given above.
\end{enumerate}

\bigskip
In the present paper, we want to extend the above two cases of graph Dirichlet and inverted Dirichlet distributions to wider families of multivariate beta-type distributions together with  respective IG maps leading to random vectors with independent generalized $\mathrm{Beta}_{\mathrm I}$ and generalized $\mathrm{Beta}_{\mathrm{II}}$ components (which would also extend the bivariate $H_{\mathrm{I}}^+$ based models, as identified in \cite{SasUoz2024} and related characterizations from \cite{KLPW25}). To achieve this goal, though, we need to restrict the family of decomposable graphs to the family of trees (then each moral DAGs $\mathcal G$ becomes a directed $r$-rooted tree $G^{(r)}$ for some $r\in V$). Recently the two graph Dirichlet-type distributions described above have been extended, keeping the assumption that the underlying graph is decomposable, in a different direction, to graph Dirichlet-multinomial and graph Dirichlet-negative-multinomial distributions in \cite[Sec. 6] {DanKol2026}  

The rest of the paper is organized as follows. In the remaining parts of Section 1, we first consider the basic bivariate case, in particular, deriving bivariate characterization, which essentially can be traced back in literature (Subsection 1.1). Then, we introduce the first and second kind tree-beta distributions and families of maps which are supposed to be IG maps for these models, and finally formulate main results of the paper: Theorem \ref{direct} and Theorem \ref{Charact} (Subsection 1.2).  In Section 2, we establish basic properties of polynomials $\Delta_G$ and $\delta_G$, as well as maps $\Psi^{(r)}$ and $\psi^{(r)}$, which are intensively used in the sequel. In Section 3, we prove the independence properties of Theorem \ref{direct} and derive the form of the normalizing constants in both tree-beta models. Section 4 is devoted to proving the characterizations of Theorem \ref{Charact}. Finally, Section 5 relates the tree-Kummer and the second kind tree-beta distributions, which allows us to find a (symmetric) explicit form of the normalizing constant of the $\mathrm{TK}^G$ law.

\subsection{Bivariate prototypes}
The simplest instance of the present construction is the two-vertex tree $(\{0, 1\}, \{\{0, 1\}\})$ with a single edge of weight~$q$.
The first-kind and second-kind  prototypes are presented in turn, as they motivate the tree-level definitions and theorems of the subsequent subsections.

We first recall the univariate distributions that appear as marginals throughout the paper and were identified in the bivariate models related to $H_{\mathrm{I}}^+$ class of IP maps in \cite{SasUoz2024} (see also \cite{KLPW25}).
Following \cite{KouVal2012}, the beta and the generalised beta distributions of the first kind, $\BetaI(\alpha,\beta)$ and $\GBI(\alpha,\beta,\gamma;q)$, $\alpha,\beta>0$, $\gamma\in\mathbb R$ and $q<1$ are defined through densities
\begin{align}
	& f_{\BetaI}(x)
	= \tfrac{\Gamma(\alpha+\beta)}{\Gamma(\alpha)\Gamma(\beta)}\,
	x^{\alpha - 1}(1 - x)^{\beta - 1} \, \1{(0, 1)}(x), \label{bI}\\
	& f_{\GBI}(x)
	=\tfrac{\Gamma(\alpha+\beta)}{\Gamma(\alpha)\Gamma(\beta)\,_2F_1(\alpha,\gamma;\alpha+\beta;q)}\,
	x^{\alpha - 1}(1 - x)^{\beta - 1}(1 - q x)^{- \gamma} \, \1{(0, 1)}(x),\label{gbI}
\end{align}
where $_2F_1$ is the Gauss hypergeometric function. The normalizing constant in \eqref{gbI} is an immediate consequence of the Euler integral representation of $_2F_1$, see e.g. formula (15.3.1) in \cite{AbrSteg1965}. We note that $\mathrm{GB}_{\mathrm I}(\alpha,\beta,\gamma;0)=\mathrm{Beta}_{\mathrm I}(\alpha,\beta)$ and additionally we define $\mathrm{GB}_{\mathrm I}(\alpha,\beta,\gamma;1)=\mathrm{Beta}_{\mathrm I}(\alpha,\beta-\gamma)$ when  $\alpha>0$ and $\beta>\gamma$.

The corresponding  beta and generalised beta distributions of the second kind, $\BetaII(\alpha,\beta)$ and $\GBII(\alpha,\beta,\gamma;q)$, the former also known under the name  beta-prime distribution, are defined through densities
\begin{align}
	& f_{\BetaII}(x)=
	\tfrac{\Gamma(\beta)}{\Gamma(\alpha)\Gamma(\beta-\alpha)}\,\tfrac{x^{\alpha - 1}}{(1 + x)^{\beta}} \, \1{(0, \infty)}(x), \label{bII} \\
	& f_{\GBII}(x) =
	\tfrac{\Gamma(\beta+\gamma)}{\Gamma(\alpha)\Gamma(\beta+\gamma-\alpha)\,_2F_1(\alpha,\gamma;\beta+\gamma;1-q)} \tfrac{x^{\alpha - 1}}{(1 + x)^{\beta}(1 + q x)^{\gamma}} \, \1{(0, \infty)}(x).\label{gbII}
\end{align}

In case of $\BetaII(\alpha,\beta)$ the parameters satisfy $\beta>\alpha>0$ and in case of $\GBII(\alpha,\beta,\gamma;q)$ the parameters satisfy $q>0$   and $\beta + \gamma>\alpha > 0$. The normalizing constant in \eqref{gbII} is given e.g. in formula (1.2) in \cite{KLPW25}. Note that  $\GBII(\alpha,\beta,\gamma;1)=\BetaII(\alpha,\beta+\gamma)$ and additionally we define $\GBII(\alpha,\beta,\gamma;0)=\BetaII(\alpha,\beta)$ when $\beta>\alpha>0$.

\bigskip
According to  \cite{PSTV2010},
\begin{align}\label{HAB}
    H_{\mathrm{I}}^{+, A, B}(x,y)
	= \paren[\Big]{\tfrac{y}{A}\,\tfrac{B + Ax + By + ABxy}{1 + x + y + Bxy},\; \tfrac{x}{B}\,\tfrac{A+Ax+By+ABxy}{1 + x+y+Axy}}, \quad x,y>0,
\end{align}
with $A,B \neq 0$, is a family of maps which is at the top of the hierarchy of $[2:2]$ quadrirational Yang-Baxter maps. Recall, see \cite{SasUoz2024}, that $H_{\mathrm{I}}^{+, A, B}$ for $A,B> 0$ are IP maps for generalized second kind beta laws. Respective  characterizations have been derived recently in \cite{KLPW25}. Some  of them use $H_{\mathrm{I}}^{+,A,\infty}$ which is defined as the pointwise limit 
\begin{equation}\label{HA}
H_{\mathrm{I}}^{+,A,\infty}(x,y):=\lim_{B\to\infty}\,H_{\mathrm{I}}^{+,A,B}(x,y)=\left(\tfrac{1+y+Axy}{Ax},\,\tfrac{xy(1+Ax)}{1+x+y+Axy}\right),\quad x,y>0
\end{equation}
and, by symmetry,
\begin{equation}\label{HB}
H_{\mathrm{I}}^{+,\infty,B}(x,y)=\left(\tfrac{xy(1+By)}{1+x+y+Bxy},\,\tfrac{1+x+Bxy}{By}\right),\quad x,y>0.
\end{equation}
We will relate these IP maps to the bivariate prototypes of the first and second kind tree beta models.

\bigskip
Let $q \leq 1$ and $a, b, \beta > 0$. Consider a bivariate random vector $K=(Y,V)$ with the density
\begin{equation}\label{fI2}
f_K(y,v)\propto y^{a - 1} v^{b - 1} (1 - y - v + qyv)^{\beta - 1} \, \1{\mathcal D}(y, v),
\end{equation}
where  $\mathcal D = \Set{(y, v) \in(0, 1)^{2} | 1 - y - v + qyv > 0}$. (Note that the function on the right-hand side of \eqref{fI2} is non-negative and integrable.) Set $X=V\, \tfrac{1 - qY}{1 - Y}$ and $U=Y\, \tfrac{1 - qV}{1 - V}$. 
Then
\begin{align}
	(U, V)
    = H_q(X,Y)
\end{align}
where
\begin{equation}\label{Hqq}
H_q(x,y) :=\paren[\Big]{y \,\tfrac{1 - qx - qy + qxy}{1 - x - qy + xy},\;x \,\tfrac{1 - y}{1 - qy}}, \quad (x,y)\in(0,1)^2.
\end{equation}
From \eqref{fI2}, by direct computation (and trivially in case $q=1$), we get
\begin{align}
	(X, Y)
	&\sim
	\BetaI(b, \beta) \otimes \GBI(a, \beta + b, b; q),\label{XYI}
	\\
	(U, V)
	&\sim
	\BetaI(a, \beta) \otimes \GBI(b, \beta + a, a; q).\label{UVI}
\end{align}
Thus $H_q$ is an IP map. 

A related independence characterization can be also deduced from a result in \cite{KLPW25}.
\begin{proposition}\label{2bI}
    Let  $X$ and $Y$ be $(0,1)$-valued, non-Dirac and independent random variables and $(U,V)=H_q(X,Y)$ for some $q<1$. Assume also that $U$ and $V$ are independent. Then there exist $a,b,\beta>0$ such that \eqref{XYI} and \eqref{UVI} hold.
\end{proposition}
\begin{proof}
Note that, see \eqref{HB},
\begin{equation}\label{Gq}
H_q=(g^{-1},g^{-1})\circ H_{\mathrm{I}}^{+,\infty,\frac1{1-q}}\circ (g,g).
\end{equation}

   Denote $W':=g(W)$ for $W\in\{X,Y,U,V\}$. Thus $X'$ and $Y'$ are independent and, in view of \eqref{Gq} we have  
   \[
   (U',V')=H_{\mathrm{I}}^{+,\infty,\frac1{1-q}}(X',Y').
   \]
   Since $U'$ and $V'$ are independent, from Theorem 3.2 of \cite{KLPW25} (note that this result is for $H_{\mathrm{I}}^{+,A,\infty}$ so one needs to rely on the symmetry $(s,t)=H_{\mathrm{I}}^{+,A,\infty}(x,y)$ iff $(t,s)=H_{\mathrm{I}}^{+,\infty,A}(y,x)$), we conclude that  
   \begin{align}
	(X', Y')
	&\sim
	\BetaII(\beta,\beta+b) \otimes \GBII(\beta+b, \beta +a, b; \tfrac1{1-q}),
	\\
	(U', V')
	&\sim
	\BetaII(\beta,\beta+a) \otimes \GBII(\beta+a, \beta +b, a; \tfrac1{1-q}).
\end{align}The final result follows on noting that if $W\sim\Beta_{\mathrm{II}}(\kappa,\lambda)$, $\lambda>\kappa>0$, then $g^{-1}(W)\sim\Beta_{\mathrm{I}}(\lambda-\kappa,\kappa)$ and if $W\sim\GBII(\kappa,
\lambda,\gamma;p)$, $p> 0$, $\lambda+\gamma>\kappa>0$ the $g^{-1}(W)\sim \GBI(\lambda+\gamma-\kappa,\kappa,\gamma,\tfrac{p-1}{p})$.
\end{proof}
We remark that when $q = 0$ then \eqref{fI2} is the density of the bivariate classical Dirichlet distribution and the independence properties generated by $H_0$ refer to the simplest version of the neutrality property of the Dirichlet distribution; see e.g. \cite{ChamLet1991, DarrochRatcliff1971}. Actually, with a suitable re-definition $Y':=1-Y$, this is equivalent to
\begin{align}
	(X, Y') \sim \BetaI(b, \beta) \otimes \BetaI(b + \beta,a),
\end{align}
and
\begin{align}
	\paren[\Big]{1 - XY', \, \tfrac{1 - Y'}{1 - XY'}}
	\sim
	\BetaI(a + \beta, b) \otimes \BetaI(a, \beta).
\end{align}

\bigskip
The second-kind analogue reads as follows. 
Let either $q> 0$, $\beta > 0$ and $a, b \in(0, \beta)$ or $q=0$, $a,b>0$, $a+b<\beta$. Consider a bivariate random vector $K=(Y,V)$ with the density
\begin{equation}
\label{fII2}
	f_K(y,v)\propto y^{a - 1} v^{b - 1} (1 + y + v + qyv)^{- \beta} \, \1{(0, \infty)^{2}}(y, v).   
\end{equation}
(Note that the function on the right-hand side above is non-negative and integrable.) Let $X:=V\tfrac{1+qY}{1+Y}$ and $U:=Y\tfrac{1+qV}{1+V}$. Then 
\begin{align}
	\label{eq:bivariate-root-switch-formula}
	(U, V)
	= h_q(X, Y),
\end{align}
where 
\begin{align}
	\label{Hq}
	h_q(x, y)
	= \paren[\Big]{y\, \tfrac{1 + qx + qy + qxy}{1 + x + qy + xy}, x\, \tfrac{1 + y}{1 + qy}},\quad x,y >0.
\end{align}
Thus from \eqref{fII2}, by direct computation (and trivially in case $q=1$), we get 
\begin{align}
	(X, Y)
	&\sim
	\BetaII(b, \beta) \otimes \GBII(a, \beta - b, b; q),\label{XY}
	\\
	(U, V)
	&\sim
	\BetaII(a, \beta) \otimes \GBII(b, \beta - a, a; q).\label{UV}
\end{align}
Hence $h_q$ is an IP map.

The  characterization based on this IP map follows from  two characterizations given in \cite{KLPW25}.

\begin{proposition}\label{2bII}
Let  $X$ and $Y$ be positive, non-Dirac and independent random variables and $(U,V)=h_q(X,Y)$ for some $1\neq q\geq 0$. Assume also that $U$ and $V$ are independent. Then there exist $\beta>0$ and $a,b\in (0,\beta)$ (in case $q=0$ additionally $\beta>a+b$) such that \eqref{XY} and \eqref{UV} hold.
\end{proposition}
\begin{proof}
Note that when $q>0$ for $x,y>0$, see \eqref{HAB},
\begin{align}
	\label{HI}
(u, v) = h_q(x, y)\quad\Leftrightarrow\quad    (u, qv) = H_{\mathrm{I}}^{+, 1, 1/q}(x, qy).
\end{align}
In view of \eqref{HI}, the definition of $(U,V)$ implies  $(U,qV)=H_{\mathrm{I}}^{+,1,1/q}(X,qY)$. Therefore 
 by the bivariate characterization of the generalized second kind beta laws related to IP maps $H^{+,A,B}$, $A,B> 0$, see Theorem 1.2 of \cite{KLPW25} for $A=1$ and $B=1/q$, we conclude that there exist $\beta>0$ and $a,b\in (0,\beta)$ such that
 \[
 (X,qY)\sim \BetaII(b, \beta) \otimes \GBII(a,  b,\beta - b; 1/q)
 \]
 and
 \[
 (U,qV)\sim \BetaII(a, \beta) \otimes \GBII(b, a, \beta - a; 1/q).
 \]
Since $W\sim\mathrm{GB}_{\mathrm{II}}(\alpha,\beta,\gamma;\kappa)$ implies $\kappa W\sim \mathrm{GB}_{\mathrm{II}}(\alpha,\gamma,\beta;1/\kappa)$ one gets  \eqref{XY} and \eqref{UV} by taking $W=qY$ or $W=qV$ and $\kappa=1/q$.

In case $q=0$ we have, see \eqref{HA},
\begin{equation}\label{HI0}
h_0=(g,g)\circ H_{\mathrm{I}}^{+,1,\infty}\circ(g,g),\quad\text{where } g(x)=1/x, \;x>0.
\end{equation}
Denote $W':=g(W)$ for $W\in\{X,Y,U,V\}$.  Thus, $X'$ and $Y'$ are independent,  $U'$ and $V'$ are independent and  $(U',V')= H_{\mathrm{I}}^{+,1,\infty}(X',Y')$. Then referring to Theorem 3.2 of \cite{KLPW25} we conclude that there exist $a,b,\beta>0$ with $\beta>a+b$ such that
   \begin{align}
	(X', Y')
	&\sim
	\BetaII(\beta-b,\beta) \otimes \BetaII(\beta-a-b, \beta -b),
	\\
	(U', V')
	&\sim
	\BetaII(\beta-a,\beta) \otimes \BetaII(\beta-a-b, \beta -a).
\end{align}
Note that $W'\sim \mathrm{Beta}_{\mathrm{II}}(A,B)$, $B>A>0$, iff $W=\tfrac1{W'}\sim \mathrm{Beta}_{\mathrm{II}}(B-A,B)$.
Applying these two observations to $X'$, $Y'$, $U'$, $V'$ yields \eqref{XY} and \eqref{UV} for $q=0$.
\end{proof}

Thus, the  bivariate characterisations theorem of \cite{KLPW25}, up-graded slightly to Propositions \ref{2bI} and \ref{2bII} above,  show that, for $q < 1$ in the first-kind case and for $q \in(0, \infty) \setminus \{1\}$ in the second-kind case, the law of $K = (Y, V)$ is uniquely determined by respective independence properties. Namely, the condition that $(X,Y)$ and $(U,V)$ have both independent components implies that $K=(Y,V)$ is distributed as the above specific distribution for each case.
The present paper extends these bivariate properties and characterisations to a multivariate tree-generated setting.

\subsection{Tree-beta models of the first and second kind}
We first introduce two graph-related polynomials which generalize $\Delta_G$ of \eqref{dDelta} and $\delta_G$  of \eqref{ddelta} in case when decomposable graph  $G$ is a tree.
\begin{definition}[Graph polynomials]
	\label{def:polynomial-pG}
	Let $G = (V, E)$ be an undirected graph and let $q = (q_{e})_{e \in E} \in \R^{E}$.
	For $k = (k_{v})_{v \in V} \in \R^{V}$, define polynomials
	\begin{align}
		&\Delta_{G}(k)
		\coloneqq
		\sum_{U \subseteq V} (-1)^{|U|}\,q^{E_U}k^{U},\\
        &\delta_{G}(k)
		\coloneqq 
		\sum_{U \subseteq V}  q^{E_U}k^{U},
	\end{align}
    with $k^{U}=\prod_{v \in U} k_{v}$ and $q^{E_U}=\prod_{e \in E_U} q_{e}$, where $E_U$ denotes the set of edges from $E$ with both ends in $U$, and the convention that $\prod_{v \in \varnothing}k_v=1$ and $\prod_{e \in \varnothing}q_e=1$ for any $k$ and $q$.

 For $U\subset V$, we  write $\Delta_U(k):=\Delta_{G_U}(k|_U)$ and $\delta_U(k):=\delta_{G_U}(k|_U)$.
\end{definition} 

Note that in case $q_e=0$ for all $e\in E$  the right-hand sides above agree with respective right-hand sides in \eqref{dDelta} and \eqref{ddelta}. 

\bigskip
 Now we introduce multivariate tree versions of univariate first and second kind beta distributions (they can be also considered as tree versions of multivariate Dirichlet and inverted Dirichlet distributions).

 \begin{definition}
Let $G=(V,E)$ be a tree with edge weights $q=(q_e)_{e\in E}$. We say that a random vector $K$ has
\begin{enumerate}
    \item the first kind tree-beta distribution, $\mathrm{TB}^{G}_{\mathrm I}(a,\beta,q)$, where $a=(a_v)_{v\in V}\in(0,\infty)^V$, $\beta>0$, $q\in(-\infty,1)^E$, is defined by the density
    \begin{equation}\label{denI}
    f_{\mathrm I}^G(k)=\,c_{\mathrm I}\, [\Delta_G(k)]^{\beta-1}\,\left(\prod_{v\in V}\,k_v^{a_v-1}\right)\,\mathbf 1_{\mathcal D}(k),
    \end{equation}
    where $c_{\mathrm I}$ is the normalizing constant and 
    \begin{align}
        \label{DG}
		\mathcal D
		\coloneqq
		\Set{k \in (0,1)^V| \Delta_U(k) > 0, \; \forall\; U \subseteq V}
	\end{align}
with $\Delta_A(k):=\Delta_{G_A}(k|_A)$.
    \item the second kind tree-beta distribution, $\mathrm{TB}^{G}_{\mathrm{II}}(a,\beta,q)$, where  $\beta >0$, $a\in(0, \beta)^V$, $q\in[0,\infty)^E$ such that $a_v + a_w < \beta$ if $q_{v, w} = 0$, is defined by the density
    \begin{equation}\label{denII}
    f_{\mathrm{II}}^G(k)=\,\tfrac{c_{\mathrm{II}} }{[\delta_G(k)]^{\beta}}\,\Biggl(\prod_{v\in V}\,k_v^{a_v-1}\Biggr)\,\mathbf 1_{(0,\infty)^V}(k),
    \end{equation}
    where $c_{\mathrm{II}}$ is the normalizing constant.
\end{enumerate}
 \end{definition}

The above distributions are well-defined, i.e. the functions \eqref{denI} and \eqref{denII} are non-negative and integrable within the given range of parameters. Probably, the easiest way to see it is through the factorizations of these densities, which follow from independence properties established in Theorem \ref{direct} together with known ranges of parameters for (generalized) $\mathrm{Beta}_{\mathrm I}$ and $\mathrm{Beta}_{\mathrm{II}}$ distributions.  

 Note also that, $\mathrm{TB}_{\mathrm I}^G$ and $\mathrm{TB}_{\mathrm{II}}^G$, when $q_e=0$ for all $e\in E$, are, respectively, the  graph Dirichlet and the graph inverted Dirichlet distributions of \cite{DKWZ25} when the related decomposable graph is a tree. Furthermore, in the case of the simplest tree with just two vertices the above densities agree, respectively, with \eqref{fI2} and \eqref{fII2}. 

 \bigskip
 To extend the independence properties of the bivariate prototypes, we introduce suitable maps which are supposed to extend maps $(Y,V)\mapsto(Y,X)$ and $(Y,V)\mapsto(U,V)$ of the bivariate cases. To construct such maps, similarly as in the MY and HV tree models, we use also directed $r$-rooted trees $G^{(r)}$ with common undirected skeleton tree $G=(V,E)$. Besides the set of children of $v\in V$ in $G^{(r)}$, we also need $\de^{(r)}(v)$, descendants of $v$, and $\debar^{(r)}(v)$, the closure of the set of descendants of $v$, which are  defined by
	\begin{align}
		\de^{(r)}(v) &:= \{w\in V\setminus \{v\}: \;\exists \text{ a directed path in }G^{(r)} \text{ from  } v \text{ to } w\}, &
		\debar^{(r)}(v) &:= \de^{(r)}(v)\cup\{v\}.
	\end{align}
	
 \begin{definition}
	\label{maps}
    	Fix an arbitrary $r \in V$. 
        \begin{enumerate}
\item For $q=(q_e)_{e\in E}\in (-\infty,1)^E$, let $\Psi^{(r)}= (\Psi_{v}^{(r)})_{v \in V} \colon \mathcal D  \to (0,1)^V$ be a map defined by       
	\begin{equation}\label{Psi}
		\Psi_{v}^{(r)}(k)
		= 1-\tfrac{\Delta_{\debar^{(r)}(v)}(k)}{\Delta_{\de^{(r)}(v)}(k)},\quad v\in V\quad\text{for } k\in\mathcal D.
        \end{equation}
        \item For $q=(q_e)_{e\in E}\in [0,\infty)^E$, let $\psi^{(r)} = (\psi_{v}^{(r)})_{v \in V} \colon (0,\infty)^V \to (0,\infty)^V$ be a map defined by
        \begin{equation} \label{psi}
        \psi_{v}^{(r)}(k)
		= \tfrac{\delta_{\debar^{(r)}(v)}(k)}{\delta_{\de^{(r)}(v)}(k)}-1,\quad v\in V\quad\text{for } k\in(0,\infty)^V.
        \end{equation}
\end{enumerate}
    \end{definition}

Lemma \ref{well_def} actually proves that the co-domains of $\Psi^{(r)}$ and $\psi^{(r)}$, as given above, are the correct ones. As we shall see in Theorem \ref{direct}, the maps $\Psi^{(r)}$ and $\psi^{(r)}$ are actually IP maps for the first and the second kind tree-beta models, respectively.  

To state the main results of this paper, we need to introduce some additional notations. Fix an arbitrary $r\in V$. For each $v \in V \setminus \{r\}$, there is a unique vertex in the directed tree $G^{(r)}$, denoted by $v_r:=\pa^{(r)}(v)$ and called the parent of $v$ in $G^{(r)}$, such that $v_r\to v$. 

\bigskip
The first set of two main results says that the maps $\Psi^{(r)}$ and $\psi^{(r)}$ are actually IP maps for tree-beta models. 

    \begin{theorem}\label{direct} Let $G=(V,E)$ be a tree with edge weights $q=(q_e)_{e\in E}$.
    \begin{enumerate}
    \item[(i)] Assume that $q\in(-\infty,1)^E$, $\beta\in(0, \infty)$ and $a\in(0, \infty)^V$. For a random vector $K$ taking values in $\mathcal{D}$,
				\begin{align}
					K \sim \mathrm{TB}^G_{\mathrm I}(a,\beta,q)
				\end{align}
                if and only if for any $r \in V$, the random vector $\Psi^{(r)}(K)=(\Psi^{(r)}_v(K))_{v\in V}$ has independent components with distributions
				\[
                \Psi^{(r)}_r(K)\sim\BetaI(a_r,\beta)
                \]
                and for $v\in V\setminus\{r\}$
                \[
					\Psi^{(r)}_v(K)
					\sim
					\GBI(a_{v}, \beta + a_{v_r}, a_{v_r}; q_{v_r,v}).
				\]            
		\item[(ii)] Assume that $q\in[0,\infty)^E$, $\beta\in(0, \infty)$ and $a\in(0, \beta)^V$ such that $a_v + a_w < \beta$ if $q_{v, w} = 0$. For a random vector $K$ taking values in $(0,\infty)^V$,
				\begin{align}
					K
					\sim \TB^{G}_{\mathrm{II}}(a,\beta,q)
				\end{align}
				if and only if for any $r \in V$, the random vector $\psi^{(r)}(K)=(\psi^{(r)}_v(K))_{v\in V}$ has independent components with distributions
				\[
                \psi^{(r)}_r(K)\sim\BetaII(a_r,\beta)
                \]
                and for $v\in V\setminus\{r\}$
                \[
					\psi^{(r)}_v(K)
					\sim
					\GBII(a_{v}, \beta - a_{v_r}, a_{v_r}; q_{v_r,v}).
				\]            
	\end{enumerate}
\end{theorem}
Note that, when $q_e=0$ for all $e\in E$, the above result reduces to Theorem 5.2 of \cite{DKWZ25} in case the decomposable graph $G$ is a tree. Furthermore, it is a natural analogue of Theorem 3.1 of \cite{MasWes2004} for the MY tree model and of Theorem 2 of \cite{PilWes16} for the HV tree model.
\begin{remark}\label{rem1}
    In view of Lemma \ref{well_def} the maps $\Psi^{(r)}$ and $\psi^{(r)}$, $r\in V$, are bijective, and thus in the sufficiency parts of both cases in Theorem \ref{direct} it suffices to have the assumption valid just for a single $r\in V$.
\end{remark}

\bigskip
The second set of two main results are characterizations of tree-beta models by independence properties generated by a restricted collections of IG maps $\Psi^{(r)}$ and $\psi^{(r)}$. 

\begin{theorem}\label{Charact}$\,$
\begin{enumerate}
    \item Assume that $q\in(-\infty,1)^E$. If $K$ is a random vector taking values in $\mathcal D$ (which is defined in \eqref{DG}) such that for every leaf $\ell\in L_G$, the components of $\Psi^{(\ell )}(K)$ are independent and non-Dirac, then there exist $\beta \in (0, \infty)$ and $a \in (0, \infty)^{V}$ such that
				\begin{align}
					K \sim \mathrm{TB}^G_{\mathrm I}(a,\beta,q).
				\end{align}
     \item    Assume that $q\in[0,\infty)^E$ and $q_{e} \neq 1$ for every $e \in E$. If $K$ is a random vector taking values in $(0, \infty)^{V}$ such that for every leaf $\ell\in L_G$, the components of $\psi^{(\ell)}(K)$ are independent and non-Dirac, then there exist $\beta \in (0, \infty)$ and $a \in (0, \beta)^{V}$ with $a_v+a_w<\beta$ whenever $q_{w,v}=0$ such that  
				\begin{align}
					K \sim \mathrm{TB}^{G}_{\mathrm{II}}(a, \beta, q).
				\end{align}        
\end{enumerate}
\end{theorem}

Note that, when $q_e=0$ for all $e\in E$, the above Theorem \ref{Charact} extends Theorem 6.6 of \cite{DKWZ25} in case the decomposable graph $G$ is a tree, since in the latter results independence properties with respect to all moral DAGs were assumed: For example, in case of the chain $G=(V,E)$ with $V=\{1,\dots,d\}$ and $E=\{\{j,j+1\},\;j=1,\dots,d-1\}$ Theorem \ref{Charact} to characterize $\mathrm{TB}^G_{\mathrm I}$ (resp. $\mathrm{TB}^{G}_{\mathrm{II}})$ law needs the independences of components of just two random vectors $\Psi^{(1)}(K)$ and $\Psi^{(d)}(K)$ (resp. $\psi^{(1)}(K)$ and $\psi^{(d)}(K)$), while in Theorem 6.6, the independences of components of additional $d-2$ random vectors $\Psi^{(2)}(K),\dots,\Psi^{(d-1)}(K)$ (resp. $\psi^{(2)}(K),\dots,\psi^{(d-1)}(K)$) are needed. Furthermore, Theorem \ref{Charact} is a natural analogue of  Theorem 4.1 of \cite{MasWes2004} for the MY tree model and of  Theorem 4 of \cite{PilWes16} for the HV tree model.

\subsection{Future projects} The tree models already known in the literature and the two we propose and study in this paper, all are multivariate versions of the bivariate IP models governed by subtraction free $[2:2]$ quadrirational Yang-Baxter maps, see \cite{SasUoz2024}. It would be plausible to use this fact for designing a common Yang-Baxter-related setup, which would cover all these cases. Moreover, these are the only known tree-related IP models. The natural question is if these are the only possible ones or there exist more. How to construct such models? That is,  multivariate random vectors with coordinates assigned to vertices of a tree $G$, together with IG maps specific for each rooted version of $G$. Until now all such models were constructed on the case-by-case basis. To search for a general scheme, the notion of the web geometry, whose relation to bivariate IP maps was recently found in \cite{Gyo2026}, looks very promising. We are planning to pursue this relation for tree models as a future project.

It is also natural to search for matrix variate versions of these tree models. At the moment a matrix version was designed only for the MY tree model in \cite{Bob2015}. There exists a matrix version of HV model but only in the bivariate case - see \cite{KolPil2020}, as well as, slightly different, in \cite{Kou2012}. In the bivariate case the IP properties of the  matrix first kind beta distributions were also studied by Olkin and Rubin, \cite{OlkRub1964}, and more recently, in \cite{HasReg2009} and \cite{Kol2016}.

For the MY tree model there exists also the Brownian motion hitting times representation, see \cite{WesWit2007}, as well as a more general representation through multivariate conditional structure of functionals of exponential Brownian motion, see \cite{MWW2009}. Existence of representations of such type, i.e. related to properties of stochastic processes, for other  tree models, remains an open problem.       

Another question worth to investigate is the relation between tree-beta models and tree-MY and tree-HV ones. In the bivariate setting the beta model, which is related to $H_I^{+}$, stands at the top of the chierarchy of quadrirational Yang-Baxter IP maps, and both $H_{II}^+$ and $H_{III}$ models can be derived by some limiting procedures from the beta one, see \cite{SasUoz2024}. In the multivariate setting we are able to follow this path only to the tree-HV model, see Section \ref{sec:limits}. Existence of such a  limiting procedure which would lead from tree-beta to tree-MY model is another open problem in this area. 

\section{Graph polynomials $\Delta_G$, $\delta_G$ and maps $\Psi^{(r)}$, $\psi^{(r)}$} \label{sec:ip-map}
 
Let us start with a trivial, but useful observation. For an undirected graph $G=(V,E)$ with edge weights $q \in \mathbb{R}^E$ and every $U\subset V$,
\begin{equation}\label{delta}
\Delta_U(k)=\delta_U(-k),\quad k\in \mathbb R^V.
\end{equation}

Actually, \eqref{delta} allows to relate $\Psi^{(r)}$ and $\psi^{(r)}$.
\begin{lemma}\label{lem:ppsi} Suppose $q\in(-\infty,1)^E$. Then, for an arbitrary $r\in V$, 
\begin{equation}\label{ppsi}
\Psi^{(r)}(-k)=-\psi^{(r)}(k),\quad k\in (-\mathcal D),
\end{equation}
where $\psi^{(r)}(k)$ is defined as in \eqref{psi} but for $k\in (-\mathcal D)$ and  $q\in(-\infty,1)^E$.
\end{lemma}
\begin{proof}
    Since $-k\in\mathcal D$ for any $U\subset V$ we have $\delta_U(k)=\Delta_U(-k)>0$ and thus the right-hand side of \eqref{psi} applies as an extended definition of $\psi^{(r)}$.  Thus, following \eqref{psi} we get 
    \[
    -\psi^{(r)}_v(k)=-\left(\tfrac{\delta_{\debar^{(r)}(v)}(k)}{\delta_{\de^{(r)}(v)}(k)}-1\right)=1-\tfrac{\Delta_{\debar^{(r)}(v)}(-k)}{\Delta_{\de^{(r)}(v)}(-k)}=\Psi^{(r)}_v(-k).
    \]
\end{proof}

\bigskip
Next, we derive a multiplicative property of polynomials $\Delta_G$ and $\delta_G$ which is a direct generalization of point (2) of Lemma 2.4 in \cite{DKWZ25}.
\begin{lemma}
	\label{product}
	Let $G=(V,E)$ be an undirected graph. If $A,B\subset V$ are disjoint and disconnected in $G$, i.e. there are no edges in $E$ with one endpoint in $A$ and the other in $B$, then
    \[
    \Delta_{A\cup B}=\Delta_A\,\Delta_B\quad \text{and}\quad \delta_{A\cup B}=\delta_A\,\delta_B.
    \]
\end{lemma}
\begin{proof}
In view of \eqref{delta} it suffices to prove only the second identity. Also, without loss of generality we can assume that $V=A\cup B$. 
	Then for any $U\subset V$ the sets $U\cap A$ and $U\cap B$ are disjoint and $E_U=E_{U\cap A}\cup E_{U\cap B}$. Clearly, $k^{U} = k^{U\cap A}k^{U\cap B}$ for any $k\in \mathbb R^V$. Therefore,
    \begin{align*}
    \delta_{A\cup B}(k)=&\delta_G(k)=\sum_{U\subset V}\,q^{E_U}\,k^U=\sum_{(U\cap A)\cup (U\cap B)\subset V}\,q^{E_{(U\cap A)\cup (U\cap B)}}\,k^{(U\cap A)\cup (U\cap B)}\\
    =&\sum_{\substack{U'\cup U''\subset V\\ U'\subset A,\,U''\subset B}}\,q^{E_{U'}}\,q^{E_{U''}}\,k^{U'}\,k^{U''}=\biggl(\sum_{U'\subset A}\,q^{E_{U'}}\,k^{U'}\biggr)\,\biggl(\sum_{U''\subset B}\,q^{E_{U''}}\,k^{U''}\biggr)=\delta_A(k)\,\delta_B(k).
    \end{align*}
\end{proof}

\bigskip
In the next result, we provide recursive representations of $\Psi^{(r)}$ (resp. $\psi^{(r)}$) in terms of $\Delta$-type (resp. $\delta$-type)  polynomials. This result extends the equivalence between Def. 2.7 and Lemma 2.9 (4) in \cite{DKWZ25} to the case where \(q_e \neq 0\), although this extension is restricted to trees (whereas the original result was established for decomposable graphs but only when \(q_e=0\) for all \(e \in E\)).
\begin{lemma}\label{pp_rep}
Fix an arbitrary root $r \in V$. Then for every $v\in V$
		\begin{equation}\label{Psi_rep}
        \Psi_{v}^{(r)}(k)
		=
		k_v \prod_{w \in \ch^{(r)}(v)} \tfrac{1 - q_{w,v} \Psi_{w}^{(r)}(k)}{1 - \Psi_{w}^{(r)}(k)},\quad k\in\mathcal D,\; q\in(-\infty,1)^E;
        \end{equation}
        and
		\begin{equation}\label{psi_rep}
		\psi_{v}^{(r)}(k)
		=
		k_{v} \prod_{w \in \ch^{(r)}(v)} \tfrac{1 + q_{w,v}\psi_{w}^{(r)}(k)}{1 + \psi_w^{(r)}(k)},\quad \begin{cases} k\in(0,\infty)^V,\,& q\in[0,\infty)^E,\\
   k\in(-\mathcal D),\,& q\in(-\infty,1)^E. 
    \end{cases}
	\end{equation}
\end{lemma}

Note that in the above lemma the quantities $\Psi_{v}^{(r)}(k)$ (resp. $\psi_{v}^{(r)}(k)$) satisfy a recurrence which goes from the leaves (with no children in $G^{(r)}$) towards the root; in particular, for any leaf $v\in L_G\setminus \{r\}$, the empty product gives $\Psi_{v}^{(r)}(k) = k_{v}$ and $\psi_{v}^{(r)}(k) = k_{v}$. Consequently, \eqref{Psi_rep} (resp. \eqref{psi_rep})  uniquely determine $\Psi^{(r)}$ (resp. $\psi^{(r)}$). Therefore both expression may serve as definitions of these maps. Note the analogy with recursive definitions of maps $\Psi^{(r)}$ for the tree models of type MY \eqref{PsiMY} and type HV \eqref{PsiHV}.

\begin{proof}
    We first prove  \eqref{psi_rep}. We assume that $k$ and $q$ are as specified in \eqref{psi_rep}. 
    
	By definition,
	\begin{align}
		\delta_{\debar^{(r)}(v)}(k) - \delta_{\de^{(r)}(v)}(k)
		= \sum_{v\in U \subseteq \debar^{(r)}(v)} q^{E_U}k^{U}.
	\end{align}
	Re-index the sum by writing $U = U' \cup \{v\}$ with $U' \subseteq \de^{(r)}(v)$.
	For such $U'$, only the variables corresponding to $\ch^{(r)}(v)$ are affected by the edges incident to $v$.
	Consequently,
	\begin{align}
		\delta_{\debar^{(r)}(v)}(k) - \delta_{\de^{(r)}(v)}(k)
		&= k_{v} \sum_{U' \subseteq \de^{(r)}(v)}
		\bigg(\prod_{w \in U' \cap \ch^{(r)}(v)}q_{v,w} \bigg) q^{E(U')}  k^{U'}.
	\end{align}
    Since $\de^{(r)}(v)$ is a  union of disjoint sets $\debar^{(r)}(w)$, $w\in \ch^{(r)}(v)$,
    \begin{align}
		\delta_{\debar^{(r)}(v)}(k) - \delta_{\de^{(r)}(v)}(k)
		&= k_{v} \prod_{w \in \ch^{(r)}(v)} \bigg( \sum_{U \subseteq \debar^{(r)}(w)} q_{v,w}^{\1{w\in U}} q^{E(U)}  k^{U} \bigg)
        \\
        &= 
        k_{v} \prod_{w \in \ch^{(r)}(v)} \Bigg( 
        \sum_{U \subseteq \de^{(r)}(w)} q^{E(U)}  k^{U}
        + q_{v,w} \sum_{w\in U \subseteq \debar^{(r)}(w)} q^{E(U)}  k^{U} 
        \Bigg)
        \\
        &= 
        k_{v} \prod_{w \in \ch^{(r)}(v)} \Big( 
        \delta_{\de^{(r)}(w)}(k)
        + q_{v,w} \big(\delta_{\debar^{(r)}(w)}(k) - \delta_{\de^{(r)}(w)}(k)\big)
        \Big).\label{identity}
	\end{align}
    
    Therefore, referring to  \eqref{psi}, and to the above identity, we see that the right-hand side of \eqref{psi_rep} can be written as
    \begin{align}
        k_v \prod_{w \in \ch^{(r)}(v)} \tfrac{
        1+ q_{v,w} \bigg(\frac{\delta_{\debar^{(r)}(w)}(k)}{\delta_{\de^{(r)}(w)}(k)} - 1 \bigg)
        }{
        1+ \bigg(\frac{\delta_{\debar^{(r)}(w)}(k)}{\delta_{\de^{(r)}(w)}(k)} - 1 \bigg)
        }
        &=
        k_v \prod_{w \in \ch^{(r)}(v)} 
        \tfrac{\delta_{\de^{(r)}(w)}(k) + q_{v,w} \Big(\delta_{\debar^{(r)}(w)}(k) - \delta_{\de^{(r)}(w)}(k) \Big)}{\delta_{\debar^{(r)}(w)}(k)}
        \\
        &= \tfrac{\delta_{\debar^{(r)}(v)}(k) - \delta_{\de^{(r)}(v)}(k)}{\delta_{\de^{(r)}(v)}(k)}=\psi^{(r)}_v(k).
    \end{align}

    To prove \eqref{Psi_rep} we consider $k\in\mathcal D$ and $q\in(-\infty,1)^E$  and rely on  \eqref{psi_rep} together with  \eqref{ppsi}. For arbitrary $r,v\in V$ we get
    \[
    \Psi^{(r)}_v(k)\stackrel{\eqref{ppsi}}{=}-\psi^{(r)}_v(-k)\stackrel{\eqref{psi_rep}}{=}-(-k_v)\prod_{w \in \ch^{(r)}(v)} \tfrac{1 + q_{w,v}\psi_{w}^{(r)}(-k)}{1 + \psi_w^{(r)}(-k)}\stackrel{\eqref{ppsi}}{=}k_v\prod_{w \in \ch^{(r)}(v)} \tfrac{1 - q_{w,v}\Psi_{w}^{(r)}(k)}{1 - \Psi_w^{(r)}(k)}.
    \]
\end{proof}

It appears that, similarly as in Lemma 2.9 (2) in \cite{DKWZ25}, polynomials $\Delta_G$ (resp. $\delta_G$) decompose nicely in terms of maps $\Psi^{(r)}$ (resp. $\psi^{(r)}$).
\begin{lemma}
	\label{lem:factorisation-skeleton}
	For any $r \in V$
	\begin{align}
    \label{Dfac}
		\Delta_{G}(k) &= \prod_{v \in V} \paren{1 - \Psi_{v}^{(r)}(k)},\quad k\in\mathcal D,\;q\in(-\infty,1)^E\\
        \delta_{G}(k) &= \prod_{v \in V} \paren{1 + \psi_{v}^{(r)}(k)},\quad \begin{cases} k\in(0,\infty)^V, & q\in[0,\infty)^E,\\
                k\in(-\mathcal D), & q\in(-\infty,1)^E.
            \end{cases}\label{dfac}
	\end{align}
\end{lemma}
\begin{remark}
    For any $U\subset V$, the same factorization applies to the induced subgraph $G_U$, equipped with the orientation inherited from $G^{(r)}$. If $G_U$ is a forest, the factorization is applied separately to its each tree and the resulting expressions are multiplied in view of  Lemma \ref{product}.
\end{remark} 

\begin{proof} It suffices to prove \eqref{dfac}. Indeed, then \eqref{Dfac} follows immediately from \eqref{delta} combined with \eqref{ppsi}.

    We will prove that for $k$ and $q$ as specified in \eqref{dfac} and for arbitrary $v\in V$ 
    \begin{align}
        \label{eq:factorisation-debar}
		\delta_{\debar^{(r)}(v)}(k) = \prod_{u \in \debar^{(r)}(v)} \paren{1 + \psi_{u}^{(r)}(k)}.
	\end{align}
    Then \eqref{dfac} would follow by taking $v=r$. 
    
    The proof of \eqref{eq:factorisation-debar} proceeds by induction on vertices from leaves to the root: if \eqref{eq:factorisation-debar} holds for all the children of a vertex, then so it does for that vertex. To start the induction we note that for a leaf $v\in L_G\setminus\{r\}$, the factorization \eqref{eq:factorisation-debar} holds since $\debar^{(r)}(v)=\{v\}$ and $\psi_{v}^{(r)}(k) = k_v$.

    For a general vertex $v$, by \eqref{psi} one has
	\begin{align}
		\delta_{\debar^{(r)}(v)}(k) 
        = (1 + \psi^{(r)}_v(k)) \, \delta_{\de^{(r)}(v)}(k).
	\end{align}
	Note that $\de^{(r)}(v)$ is the union of disjoint  sets  $\debar^{(r)}(w)$ over $w \in \ch^{(r)}(v)$, which are mutually disconnected in $G$ as there are no edges in $G$ connecting $\debar^{(r)}(w)$ and $\debar^{(r)}(w')$ for $w \neq w'$. 
	Therefore, Lemma \ref{product} yields
	\begin{equation*}
		\delta_{\de^{(r)}(v)}(k) = \prod_{w \in \ch^{(r)}(v)} \,\delta_{\debar^{(r)}(w)}(k).
	\end{equation*}
	The induction hypothesis is that \eqref{eq:factorisation-debar} holds for all $w\in\mathfrak{ch}^{(r)}(v)$, whence 
	\begin{align}
		\delta_{\de^{(r)}(v)}(k) 
        = \prod_{w \in \ch^{(r)}(v)} \prod_{u \in \debar^{(r)}(w)} \paren{1 + \psi^{(r)}_u(k)}
        = \prod_{u \in \de^{(r)}(v)} \paren{1 + \psi^{(r)}_u(k)}.
	\end{align}
	Multiplying by the factor $(1 + \psi^{(r)}_v(k))$ proves \eqref{eq:factorisation-debar}.
\end{proof}

Now we are ready to prove that $\Psi^{(r)}$ and $\psi^{(r)}$ are bijections with domains and images as specified in Definition \ref{maps}. 

\begin{lemma}\label{well_def}
    Fix an arbitrary $r\in V$. The maps $\Psi^{(r)}$ and $\psi^{(r)}$ as given in \eqref{Psi} and \eqref{psi} are  bijections with inverse functions $\Phi^{(r)}=(\Phi^{(r)}_v)_{v\in V}$ and $\phi^{(r)}=(\phi^{(r)}_v)_{v\in V}$ given by
    \begin{equation}\label{Phi}
    \Phi^{(r)}_v(x)=x_{v} \prod_{w \in \ch^{(r)}(v)} \frac{1 - x_w}{1 - q_{w,v}x_w},\quad x\in(0,1)^V,\;v\in V,\;\text{when }q\in(-\infty,1)^E,
    \end{equation}
    and 
    \begin{equation}\label{phi}
    \phi^{(r)}_v(x)=x_{v} \prod_{w \in \ch^{(r)}(v)} \frac{1 + x_w}{1 + q_{w,v}x_w},\quad  x\in(0,\infty)^V,\;v\in V,\,\text{when }q\in[0,\infty)^E.
    \end{equation}
    
    For an arbitrary $r\in V$
    \begin{equation}\label{pphi}
    \Phi^{(r)}(-x)=-\phi^{(r)}(x),\quad x\in (-1,0)^V,
    \end{equation}
    where $\phi^{(r)}_v(x)$, $v\in V$, is defined as in \eqref{phi} but for $x\in(-1,0)^V$ and $q\in(-\infty,1)^E$.
    \end{lemma}
    \begin{proof}
Fix an arbitrary $r\in V$. 

To prove that $\psi^{(r)}$ is bijective with the inverse function $\phi^{(r)}$, first, we take $k\in(0,\infty)^V$. Then,  \eqref{psi_rep}  applied recursively starting  from the leaves, gives $\psi^{(r)}(k)\in(0,\infty)^V$. Furthermore, for any $v\in V$
\[
\phi^{(r)}_v(\psi^{(r)}(k))=\psi^{(r)}_v(k)\,\prod_{w\in\mathfrak{ch}^{(r)}(v)}\,\tfrac{1+\psi^{(r)}_w(k)}{1+q_{w,v}\psi^{(r)}_w(k)}=k_v,
\]
where we use \eqref{phi} and then \eqref{psi_rep}. Consequently, $\phi^{(r)}\circ \psi^{(r)}=\mathrm{id}_{(0,\infty)^V}$.

Next, take $x\in (0,\infty)^V$. Then, from \eqref{phi}, $k(x):=\phi^{(r)}(x)\in(0,\infty)^V$ and for any $v\in V$ 
    \begin{equation}\label{kx}
    k_v(x)=x_v\,\prod_{w\in\mathfrak{ch}^{(r)}(v)}\,\tfrac{1+x_w}{1+q_{w,v}x_w}.
    \end{equation}
   Now, we prove inductively, starting from $v\in L_G\setminus\{r\}$ (then $\mathfrak{ch}^{(r)}(v)=\varnothing$) and proceeding towards the root, that $\psi^{(r)}_v(k(x))=x_v$. When $v\in L_G\setminus\{r\}$, it is obvious. For $v\not\in L_G\setminus\{r\}$, by the induction assumption, $\psi_w^{(r)}(k(x))=x_w$ for every $w\in\mathfrak{ch}^{(r)}(v)$. Then  
   \[
   \psi^{(r)}_v(k(x))\overset{\eqref{psi_rep}}{=}k_v(x)\,\prod_{w\in \mathfrak{ch}^{(r)}(v)}\,\tfrac{1+q_{w,v}\psi_w^{(r)}(k(x))}{1+\psi_w^{(r)}(k(x))}=k_v(x)\,\prod_{w\in \mathfrak{ch}^{(r)}(v)}\,\tfrac{1+q_{w,v}x_w}{1+x_w}\overset{\eqref{kx}}{=}x_v.
   \]
   Consequently, $\psi^{(r)}\circ \phi^{(r)}=\mathrm{id}_{(0,\infty)^V}$.

    Thus, $\phi^{(r)}$ is the inverse function of $\psi^{(r)}$.

\bigskip
We prove that $\Psi^{(r)}$ is bijective with the inverse function $\Phi^{(r)}$, by repeating the above argument with obvious modifications: $\Psi^{(r)}$ in place of $\psi^{(r)}$, $\Phi^{(r)}$ in place of $\phi^{(r)}$, respective changes of domains and codomains and references to \eqref{Psi_rep} and \eqref{Phi} instead of \eqref{psi_rep} and \eqref{phi}. The only part which needs explanation is  that  $k(x):=\Phi^{(r)}(x)\in \mathcal D$ where
\begin{equation}\label{kvx}
k_v(x)=x_v\,\prod_{w\in\mathfrak{ch}^{(r)}(v)}\,\tfrac{1-x_w}{1-q_{w,v}x_w},\quad x\in(0,1)^V.
\end{equation}
Thus, it remains to prove that $\Delta_U(k(x))>0$ for every non-empty $U\subseteq V$ (note that $\Delta_\varnothing=1$).

\bigskip
To this end we fix an arbitrary non-empty $U\subseteq V$ and consider the induced forest $G_U$. Let $r\in V$, and equip $G_U$ with the orientation inherited from $G^{(r)}$. Thus, denoting by $\mathfrak{ch}(v)$ the set of children of $v$ in $G_U$ under this orientation, we have $\mathfrak{ch}(v)=\mathfrak{ch}^{(r)}(v)\cap U$ for any $v\in U$. For each $v\in U$, recursively define
\begin{equation}\label{eq:yH}
y_v^U
:=k_v(x)\prod_{w\in\ch(v)}
\frac{1-q_{w,v}y_w^U}{1-y_w^U}.
\end{equation}

\bigskip
We first prove by induction starting with leaves of $G_U$ with no children and moving up-arrows of directed $G_U$ that
\begin{equation}\label{eq:y-bound}
0<y_v^U\le x_v<1, \quad v\in U.
\end{equation}
If $v\in U$ is such that $\mathfrak{ch}(v)=\varnothing$ then $y_v^U=k_v(x)$. Since \eqref{kvx} yields $0<k_v(x)\le x_v$ the result follows. 

For the induction step we fix $v\in U$ such that $\mathfrak{ch}(v)\neq\varnothing$ and note that $0<y_w^U\le x_w$ for all $w\in \mathfrak{ch}(v)$. Then, using the function $R_q(t):=\frac{1-qt}{1-t}>1$, $0<t<1$, we can write
\begin{align*}
y_v^U=k_v(x)\prod_{w\in\ch(v)}\,R_{q_{w,v}}(y_w^U)\le k_v(x)\prod_{w\in\ch(v)}\,R_{q_{w,v}}(x_w)\le k_v(x)\prod_{w\in\ch^{(r)}(v)}\,R_{q_{w,v}}(x_w)=x_v,
\end{align*}
where the first inequality follows from the induction assumption and the fact that $R_{q_{w,v}}$ is increasing since $q_{w,v}<1$, while the second inequality uses $R_{q_{w,v}}(x_w)>1$, which follows from $q_{w,v}<1$ and $x_w>0$ (clearly, $\ch(v)\subset \ch^{(r)}(v)$). The first equality above yields $y_v^U>0$. Thus \eqref{eq:y-bound} holds.

\bigskip
In view of Lemma \ref{product}, to prove that $\Delta_U(k(x))>0$ it suffices to prove that $\Delta_C(k(x))>0$ for each connected component $G_C$ of $G_U$.  Fix such $C$ and consider a directed tree with skeleton $G_C$ (equipped with the orientation induced by $G^{(r)}$). For such a directed tree we prove inductively (moving up-arrow from the leaves of $G_C$)  that for each $v\in C$ the following inequalities hold: $\Delta_{\debar(v)}(k(x))>0$, $\Delta_{\de(v)}(k(x))>0$, and
\begin{equation}\label{eq:ratio-y}
y^U_v=1-\frac{\Delta_{\debar(v)}(k(x))}{\Delta_{\de(v)}(k(x))}
\end{equation}
where $\de(v)$ is the descendants of $v$ in $G_C$, and $\debar(v)$ is its closure.
If $\mathfrak{ch}(v)=\varnothing$ then $\Delta_{\de(v)}(k(x))=1$ and $\Delta_{\debar(v)}(k(x))=1-k_v(x)>0$, which also proves \eqref{eq:ratio-y}. Take now $v\in C$ such that $\mathfrak{ch}(v)\neq\varnothing$. Then by induction assumption $\Delta_{\debar(w)}(k(x))>0$ for each $w\in\mathfrak{ch}(v)$, which gives
$\Delta_{\de(v)}(k(x))=\prod_{w\in\mathfrak{ch}(v)}\,  \Delta_{\debar(w)}(k(x))>0$. Referring again to  the induction assumption and \eqref{eq:yH} we get
\begin{multline}
y_v^U=k_v(x)\prod_{w\in\mathfrak{ch}(v)}\,\tfrac{\Delta_{\de(w)}(k(x))}{\Delta_{\debar(w)}(k(x))}\,\left(1-q_{w,v}\left(1-\tfrac{\Delta_{\debar(w)}(k(x))}{\Delta_{\de(w)}(k(x))}\right)\right)\\
=\tfrac{k_v(x)}{\Delta_{\de(v)}(k(x))}\,\prod_{w\in\mathfrak{ch}(v)}\,\left(\Delta_{\de(w)}(k(x))(1-q_{w,v})+q_{w,v}\Delta_{\debar(w)}(k(x))\right)
=\tfrac{\Delta_{\de(v)}(k(x))-\Delta_{\debar(v)}(k(x))}{\Delta_{\de(v)}(k(x))},
\end{multline}
where in the last equality we used the representation \eqref{delta} and referred to the identity \eqref{identity}. So, \eqref{eq:ratio-y} holds true.
 
 Since $y_v^U<1$ by \eqref{eq:y-bound}, we finally have 
 \[
 \Delta_{\debar(v)}(k(x))=\Delta_{\de(v)}(k(x))(1-y_v^U)>0,\quad v\in C.
 \]
In particular, taking $a\in C$ such that $C=\debar(a)$ we get $\Delta_C(k(x))=\Delta_{\debar(a)}(k(x))>0$.

    \bigskip
To prove \eqref{pphi} take $r\in V$ and $x\in(-1,0)^V$. Then 
\[
\Phi^{(r)}_v(-x)=(-x_v)\,\prod_{w\in\mathfrak{ch}^{(r)}(v)}\,\tfrac{1+x_w}{1+q_{w,v}x_w}=-\phi^{(r)}_v(x),\quad v\in V.
\]
    \end{proof}

The factorisation in \ref{lem:factorisation-skeleton} and the recursive description in Lemma \ref{pp_rep} form the algebraic core of the IP property.
The next result records the effect of changing the root across a single edge.

\subsection{Changing the root across an edge} \label{subsec:root-shift}

If $r,s\in V$ are adjacent in $G$, then $G^{(r)}$ and $G^{(s)}$ differ by the orientation of the edge $\{r,s\}$ only. Consequently, the coordinate systems $\Psi^{(r)}$ and $\Psi^{(s)}$ are only locally changed by a fixed bivariate rational map parametrized by $q_{r,s}$.
The relation with $H_{\mathrm{I}}^+$ class of the Yang--Baxter hierarchy, \cite{SasUoz2024}, refers to \eqref{Gq}, \eqref{HI} and \eqref{HI0}.
\begin{lemma}[Local form of the root change]
	\label{rootchange}
	Let $r, s \in V$ be adjacent vertices in $G$.
    \begin{enumerate}
    \item Let  $q_{r,s}\in(-\infty,1)$. Let $k \in \mathcal D$ and set $x \coloneqq \Psi^{(r)}(k)$ and $y \coloneqq \Psi^{(s)}(k)$.
	Then
	\begin{align}
		\label{eq:root-shift-local1}
		y_{v} &= x_{v} \quad v \in V\setminus \{r, s\}, &
		x_{r} &= y_{r} \tfrac{1 - q_{r,s} \, x_{s}}{1 - x_{s}}, &
		y_{s} &= x_{s} \tfrac{1 - q_{r,s} \, y_{r}}{1 - y_{r}}
	\end{align}
    and with $H_q$ defined in \eqref{Hqq} 
    \begin{align}
        (y_s, y_r) &= H_{q_{r,s}}(x_r, x_s), &
        (x_r, x_s) &= H_{q_{r,s}}(y_s, y_r).
    \end{align}
    \item Let  $q_{r,s}\in[0,\infty)$. 	Let $k \in (0,\infty)^V$ and set $x \coloneqq \psi^{(r)}(k)$ and $y \coloneqq \psi^{(s)}(k)$.
	Then
	\begin{align}
		\label{eq:root-shift-local}
		y_{v} &= x_{v} \quad v \in V\setminus\{r, s\}, &
		x_{r} &= y_{r} \tfrac{1 + q_{r,s} \, x_{s}}{1 + x_{s}}, &
		y_{s} &= x_{s} \tfrac{1 + q_{r,s} \, y_{r}}{1 + y_{r}}
	\end{align}
    and with $h_q$ defined in \eqref{Hq}
    \begin{align}
        (y_s, y_r) &= h_{q_{r,s}}(x_r, x_s), &
        (x_r, x_s) &= h_{q_{r,s}}(y_s, y_r).
    \end{align}
    
    \end{enumerate}
\end{lemma}
\begin{proof}
Note that the orientation of $G^{(r)}$ and $G^{(s)}$ agrees on every edge except $\{r,s\}$. Therefore, by \eqref{Psi} and \eqref{psi}, $\Psi^{(r)}_v=\Psi^{(s)}_v$ in case (1)  and $\psi^{(r)}_v=\psi^{(s)}_v$ in case (2) for $v\in V\setminus\{r,s\}$. Furthermore, in view of \eqref{Psi_rep} and \eqref{psi_rep}
\[
x_r=k_r\tfrac{1\pm q_{r,s}x_s}{1\pm x_s}\,A, \quad x_s=k_sB,\quad y_r=k_rA,\quad y_s=k_s\tfrac{1\pm q_{r,s}y_r}{1\pm y_r}B,
\]
where $A=\prod_{v\in \mathfrak{nb}(r)\setminus\{s\}}\tfrac{1\pm q_{v,r}z_v}{1\pm z_v}$ and $B=\prod_{v\in \mathfrak{nb}(s)\setminus\{r\}}\tfrac{1\pm q_{v,r}z_v}{1\pm z_v}$ and $z_v$ denotes the common value of $x_v=y_v$. Here $\mathfrak{nb}(v)=\{w\in V:\,\{w,v\}\in E\}$. For $\pm$ everywhere above we choose $-$ in case (1) and $+$ in case (2). Hence the result follows.
\end{proof}

\section{Forward direction: tree-beta laws imply independence} \label{sec:tb-ip}
\subsection{Proof of Theorem \ref{direct}}
\begin{proof}
$\,$ Recall the notation: for $r\in V$ and $v\in V\setminus \{r\}$, we denote $v_r=\mathfrak{pa}^{(r)}(v)$.

\begin{enumerate}
    \item[(ii)]
	Fix $r \in V$ and write $x = \psi^{(r)}(k)$ for $k \in (0,\infty)^V$.
	By \eqref{phi} we see that $k_{v}(x)=\phi^{(r)}_v(x)$ depends only on $(x_w,\,w\in\debar^{(r)}(v))$. Therefore $\partial k_{v}(x)/ \partial x_w = 0$ for $w\not\in \debar^{(r)}(v)$. By ordering the vertices according to the descendant relation induced by $G^{(r)}$, the Jacobian matrix $\tfrac{\partial k}{\partial x}$ is triangular having $\tfrac{\partial k_{v}(x)}{ \partial x_v}$, $v\in V$, on its diagonal. Thus
	\begin{align}
		\label{eq:jacobian-x-u}
		\det \tfrac{\partial k}{\partial x}
		= \prod_{v \in V}\,\tfrac{\partial k_v(x)}{\partial x_v} =\prod_{v \in V}\, \prod_{w \in \ch^{(r)}(v)} \tfrac{1 + x_{w}}{1 + q_{w,v}x_{w}}
		= \prod_{v \in V \setminus \{r\}} \tfrac{1 + x_{v}}{1 + q_{v_r,v}x_{v}}.
	\end{align}
    If $K\sim\TB_{\mathrm{II}}^G(a,\beta,q)$, then the random vector $\psi^{(r)}(K)$ has a density $g$, which can be derived from the density \eqref{denII} by the change of variables formula \eqref{kx}. We apply the above Jacobian as well as \eqref{dfac} and get
	\begin{align}\nonumber
		g(x)&=c_{\mathrm{II}}\Biggl(\prod_{v \in V \setminus \{r\}} \tfrac{1 + x_{v}}{1 + q_{v_r,v}x_{v}}\Biggr)\,\tfrac1{\bigl(\prod_{v\in V}(1+x_v)\bigr)^\beta}
		\Biggl(\prod_{v \in V} \Bigl(x_v\prod_{w\in\mathfrak{ch}^{(r)}(v)}\,\tfrac{1+x_w}{1+q_{v,w}x_w}\Bigr)^{a_v-1}\Biggr)\,\1{(0,\infty)^V}(x) \\
        &=c_{\mathrm{II}}\tfrac{x_r^{a_r-1}}{(1+x_r)^{\beta}}\1{(0,\infty)}(x_{r})\,\prod_{v\in V\setminus \{r\}}\, 
		\biggl(\tfrac{x_{v}^{a_{v} - 1}}{(1 + x_{v})^{\beta - a_{v_r}} \, (1 + q_{v_r,v}x_{v})^{a_{v_r}}}
		\, \1{(0,\infty)}(x_{v})\biggr),\label{gII}
        \end{align}
        which is the density of the claimed product law.

    Conversely, since $\psi^{(r)}$ is a bijection, the law of $\psi^{(r)}(K)$ determines uniquely the law of $K$. Hence, the claimed product law implies $K\sim\mathrm{TB}_{\mathrm{II}}(a,\beta,q)$.
\item[(i)]
	The proof in this case essentially follows the lines above. The only change is to take $x\in(0,1)^V$ and to replace all $+$'s with $-$'s, while making the change of variables \eqref{kvx} in the density \eqref{denI}, applying \eqref{Dfac} and computing the Jacobian. It gives the density $g$ of $\Psi^{(r)}(K)$ of the form
    \begin{align}
    &g(x)=c_{\mathrm I}\, x_r^{a_r-1}(1-x_r)^{\beta-1}\1{(0,1)}(x_{r}) \nonumber\\
    &\times\,\prod_{v\in V\setminus \{r\}}\, 
		\biggl(x_{v}^{a_{v} - 1}(1 - x_{v})^{\beta + a_{v_r}-1} \, (1 - q_{v_r,v}x_{v})^{-a_{v_r}}
		\, \1{(0,1)}(x_{v})\biggr), \label{gI}
    \end{align}
    which, up to the normalizing constant, is the density of the claimed product law.

    Conversely, since $\Psi^{(r)}$ is a bijection, the law of $\Psi^{(r)}(K)$ determines uniquely the law of $K$. Hence, the claimed product law implies $K\sim\mathrm{TB}_{\mathrm I}(a,\beta,q)$.
   \end{enumerate} 
\end{proof}

\subsection{Normalizing constants}

The normalizing constants of $\mathrm{TB}_{\mathrm I}^G$ and $\mathrm{TB}_{\mathrm{II}}^G$ distributions can be obtained from the independence statements of Theorem \ref{direct}.
\begin{corollary} $\,$
	\begin{enumerate}
	    \item[(i)] Assume that $q \in (-\infty, 1)^E$, $\beta>0$ and $a\in(0, \infty)^V$.
	Then $\TB^{G}_{\mathrm{I}}(a, \beta, q)$ is a probability measure, and the normalizing constant $c_{\mathrm I}$ takes the form
	\begin{align}
    c_{\mathrm I}=\tfrac1{\Gamma(\beta)\,\prod_{v\in V}\,\bigl[\Gamma(a_v)\Gamma(a_v+\beta)^{d_v-1}\bigr]}\,\prod_{\{u,w\}\in E}\,
    \tfrac{\Gamma(a_u+a_w+\beta)}{{}_{2}F_{1}(a_{u}, a_{w}; a_{u} + a_{w} + \beta; q_{u, w})},\label{cI}
	\end{align}
    where $d_v=|\{u\in V:\,\{u,v\}\in E\}|$ is the degree of the vertex $v\in V$.
        \item[(ii)] Assume that $q \in[0, \infty)^E$, $a \in (0, \beta)^V$, and that $\beta > a_{u} + a_{v}$ whenever $q_{u, v} = 0$.
	Then $\TB^{G}_{\mathrm{II}}(a, \beta, q)$ is a probability measure and the normalizing constant $c_{\mathrm{II}}$ takes the form
	\begin{align}
   c_{\mathrm{II}}=\tfrac{\Gamma(\beta)^{|V|}}{\prod_{v\in V}\,\bigl[\Gamma(a_v)\Gamma(\beta-a_v)\bigr]}
    \,\tfrac1{\prod_{\{u,w\}\in E}\,{}_{2}F_{1}(a_{u}, a_{w}; \beta; 1 - q_{u, w})}.\label{cII}
	\end{align}
	\end{enumerate}	
\end{corollary}
\begin{proof}$\,$ Recall again the notation: for $r\in V$ and $v\in V\setminus \{r\}$, we denote $v_r=\mathfrak{pa}^{(r)}(v)$.
    \begin{enumerate}
        \item[(i)] In view of \eqref{gI} the constant $c_{\mathrm{I}}$ can be computed explicitly as a product of normalizing constants as given in \eqref{bI}  for $\mathrm{Beta}_{\mathrm I}$ and  \eqref{gbI} for $\mathrm{GB}_{\mathrm I}$ distribution from the first part of Theorem \ref{direct}. Therefore, for any $r\in V$
        \[
        c_{\mathrm I}=\tfrac{\Gamma(a_r+\beta)}{\Gamma(\beta)\Gamma(a_r)}\,\prod_{v\in V\setminus\{r\}}\,\tfrac{\Gamma(a_v+a_{v_r}+\beta)}{\Gamma(a_v)\Gamma(a_{v_r}+\beta)\,_2F_1(a_{v_r},a_v;a_v+a_{v_r}+\beta;q_{v_r,v})}.
        \]
        Rearranging the factors into products over vertices and edges gives \eqref{cI}. 
        
        \item[(ii)] In view of \eqref{gII} the constant $c_{\mathrm{II}}$ can be computed  explicitly as a product of normalizing constants as given in \eqref{bII}  for $\mathrm{Beta}_{\mathrm{II}}$ and  \eqref{gbII} for $\mathrm{GB}_{\mathrm{II}}$ distribution from the second part of Theorem \ref{direct}. Therefore, for any $r\in V$
        \[
        c_{\mathrm{II}}=\tfrac{\Gamma(\beta)}{\Gamma(a_r)\Gamma(\beta-a_r)}\,\prod_{v\in V\setminus\{r\}}\,\tfrac{\Gamma(\beta)}{\Gamma(a_v)\Gamma(\beta-a_v)\,_2F_1(a_{v_r},a_v;\beta;1-q_{v_r,v})}.
        \]
        Rearranging the factors into products over vertices and edges gives \eqref{cII}. 
    \end{enumerate}
\end{proof}

\section{Converse direction: independence imply tree beta laws} \label{sec:characterisation}
\label{sec:bivariate}
This section is devoted to prove Theorem \ref{Charact}.

\begin{proof}[Proof of Theorem \ref{Charact} (ii)]
Set	$X^{(r)} \coloneqq \psi^{(r)}(K)$, $r \in V$. Then Lemma \ref{well_def} implies $\phi^{(r)}(X^{(r)}) = K$, $r\in V$. The proof proceeds by induction on $|E|$. 

\bigskip
	Assume $|E|=1$ and denote $V = \{\ell, p\}$. Then Lemma \ref{rootchange} implies $(X^{(p)}_p, X^{(p)}_\ell) = h_{q_{\ell, p}}(X^{(\ell)}_\ell, X^{(\ell)}_p)$. 
    By the assumption, since $\ell$ and $p$ are both leaves in $V$, $X^{(p)}_p$ and $X^{(p)}_\ell$ are independent, $X^{(\ell)}_\ell$ and $X^{(\ell)}_p$ are independent, and all four random variables are non-Dirac and positive. 
    Hence Proposition \ref{2bII}  yields 
    \begin{align}
        \label{eq:distribution-bivariate}
        \left.
        \begin{alignedat}{2}
		(X_{p}^{(p)}, X_{\ell}^{(p)})
		&\sim \Beta_{\mathrm{II}}(a_{p}, \beta)
		&&\otimes \GB_{\mathrm{II}}(a_{\ell}, \beta - a_{p}, a_{p}; q_{\ell, p}), \\
		(X_{\ell}^{(\ell)}, X_{p}^{(\ell)})
		&\sim \Beta_{\mathrm{II}}(a_{\ell}, \beta)
		&&\otimes \GB_{\mathrm{II}}(a_{p}, \beta - a_{\ell}, a_{\ell}; q_{\ell, p})
        \end{alignedat}
        \right.
	\end{align}  
    for some parameters $\beta>0$ and $a_p,a_\ell\in(0,\beta)$ (and $\beta>a_\ell+a_p$ when $q_{\ell,p}=0$). Thus, by Theorem \ref{direct} it follows that 
    \[
    K=\phi^{(\ell)}(X_\ell^{(\ell)},X_p^{(\ell)})\sim \mathrm{TB}_{\mathrm{II}}^G(a,\beta,q),
    \]
    where $a=(a_\ell,a_p)$ and $q=(q_{\ell,p})$.
    
  \bigskip
	Assume now that $|E| > 1$, and let $\ell\in L_G$. 
    Then there exists the unique $p\in V$ such that $\{p,\ell\}\in E$. Denote $\check V = V \setminus \{\ell\}$ and $\check G = G_{\check V}$. Clearly, $|E_{\check V}|=|E|-1$.
    Let $\check{K}$ be defined by
	\begin{align}
		\check K_{p} &:= K_{p} \frac{1 + q_{\ell, p} K_{\ell}}{1 + K_{\ell}},
		&
		\check K_{u} &:= K_{u}, \quad u \in \check{V} \setminus \{p\}.
	\end{align}
    Clearly, $\check K$ takes values in $(0,\infty)^{\check V}$. Furthermore, for the maps $\check \psi^{(r)}$, $r\in \check V$, associated with $\check G^{(r)}$ and $\check q = q|_{E(\check V)}$, as in Definition \ref{maps}, the recursion \eqref{psi_rep} gives 
	\begin{align}\label{cxr}
		\check X^{(r)} \coloneqq \check \psi^{(r)}(\check K) = X^{(r)}|_{\check V}\quad r\in \check V.
	\end{align} 

	To apply induction assumption to $\check K$, we need to verify that components of $\check X^{(w)}$ are non-Dirac and independent for every $w\in L_{\check G}$. 
    
    For $w=\ell'\in L_{\check G}\setminus \{p\}$, since  $\ell'\in L_G$, it follows directly from \eqref{cxr}. In case $p\in L_{\check G}$, we need to study components of $\check X^{(p)}=X^{(p)}|_{\check V}$.
    
    By Lemma \ref{rootchange}, $X_{u}^{(p)} = X_{u}^{(\ell)}$ for all $u \in \check{V}\setminus \{p\}$ (which implies that  $X_u^{(p)}$, $u \in \check{V}\setminus \{p\}$, are non-Dirac and independent) and 
    \begin{align}
        \label{xpp}
        X^{(p)}_p=h_{q_{p,\ell}}^{(1)}\bigl(X^{(\ell)}_{\ell},X^{(\ell)}_{p}\bigr),
    \end{align}
     where $h_q^{(1)}$ is the first coordinate of $h_q$. Thus, $X^{(p)}_p$ is non-Dirac since $X^{(\ell)}_\ell$ and  $X^{(\ell)}_p$ are independent and non-Dirac. Furthermore, \eqref{xpp} implies that $X^{(p)}_p$ and $(X^{(\ell)}_u)_{u\in \check V\setminus \{p\}}$ are independent. Hence, by \eqref{cxr},  coordinates of $\check X^{(p)}$ are independent and non-Dirac.

	Therefore by the induction hypothesis applied to $\check K$ (and $\check G$), there exist  parameters  $\beta>0$ and $\check a=(a_v)_{v \in \check V}\in(0,\beta)^{\check V}$ with $a_w+a_v<\beta$ whenever $q_{w,v}=0$, $\{w,v\}\in E_{\check V}$, such that
	\begin{align}
		\check K \sim \TB^{\check G}_{\mathrm{II}}(\check a, \beta, \check q).
	\end{align}
   Therefore, Theorem \ref{direct} yields that for any $r\in\check V$
	\begin{align}
        \label{eq:distribuion-Xr}
        X^{(r)}|_{\check V} =
		\check X^{(r)}
		\sim
		\Beta_{\mathrm{II}}(a_r,\beta)\otimes\bigotimes_{v \in \check V\setminus\{r\}} \GB_{\mathrm{II}} (a_v, \beta - a_{v_r}, a_{v_r}; q_{v_r,v}),
	\end{align}
    where $v_r=\mathfrak{pa}(v)$ in $\check G^{(r)}$ (note that $\mathfrak{pa}(v)=\mathfrak{pa}^{(r)}(v)$ for $v\in\check V$).
	Since $|E| > 1$, there exists an $\ell'\in \check V$, such that $\ell'\in (L_{\check G}\cap L_G)\setminus\{p\}$.
	In view of \eqref{cxr} we have $X^{(p)}_p = \check X^{(p)}_p$.  Thus Lemma \ref{well_def} implies
    \begin{align}\label{xpp2}
        X^{(p)}_p  = \check \psi^{(p)}_p\left(\check\phi^{(\ell')}\left(\check X^{(\ell')}\right)\right).
    \end{align}
    Since $\ell$ has no children in $G^{(p)}$ and in $G^{(\ell')}$, it follows that $X^{(p)}_{\ell} =K_\ell= X_{\ell}^{(\ell')}$. Since, by the assumption, components of $X^{(\ell')}$ are independent, \eqref{cxr} implies that $X^{(p)}_{\ell}$ and $\check X^{(\ell')}$ are independent. Hence \eqref{xpp2} yields the independence of $X_{\ell}^{(p)}$ and  $X_{p}^{(p)}$. 
    On the other hand, the pair $X_{\ell}^{(\ell)},$ and $X_{p}^{(\ell)}$ are independent and  non-Dirac because $\ell \in L_G$. 
    Referring again to Lemma \ref{rootchange}, we get $(X_p^{(p)},X^{(p)}_\ell)=h_{q_{p,\ell}}(X_{\ell}^{(\ell)},X_p^{(\ell)})$. Hence, by  Proposition \ref{2bII},  there exist parameters $\bar \beta>0$ and $\bar a_p, a_{\ell}\in(0,\bar \beta)$ (additionally, $a_\ell+\bar a_p<\bar \beta$ when $q_{\ell,p}=0$)  such that 
        \begin{alignat}{2}
        \label{eq:distribuion-Xp-Xell}
		(X_{p}^{(p)}, X_{\ell}^{(p)})
		&\sim \Beta_{\mathrm{II}}(\bar a_{p}, \bar \beta)
		&&\otimes \GB_{\mathrm{II}}(a_{\ell}, \bar \beta - \bar a_{p}, \bar a_{p}; q_{\ell, p}), \\
		(X_{\ell}^{(\ell)}, X_{p}^{(\ell)})
		&\sim \Beta_{\mathrm{II}}(a_{\ell}, \bar \beta)
		&&\otimes \GB_{\mathrm{II}}(\bar a_{p}, \bar \beta -  a_{\ell}, a_{\ell}; q_{\ell, p}).
	\end{alignat}

    Comparing the representations of the distribution of $X^{(p)}_p$ that appear in \eqref{eq:distribuion-Xr} and \eqref{eq:distribuion-Xp-Xell}, one obtains that $\beta = \bar \beta$, $a_p =  \bar a_p$.
    
  Since $X^{(\ell)}$ has independent coordinates, the  random vectors
  $$
  (X_p^{(p)},X_\ell^{(p)})=h_{q_{p,\ell}}(X_{\ell}^{(\ell)},X_p^{(\ell)})\quad\text{and}\quad (X_{u}^{(p)})_{u\in V\setminus\{p,\ell\}} = (X_{u}^{(\ell)})_{u\in V\setminus \{p,\ell\}}
  $$ 
  are independent.  Therefore,   $X^{(p)}$ has independent components, whence, combining \eqref{eq:distribuion-Xr} for $r=p$ with \eqref{eq:distribuion-Xp-Xell}   we get 
	\begin{align}
		X^{(p)}
        \sim
		\Beta_{\mathrm{II}}(a_p,\beta)\otimes\,\bigotimes_{v \in V\setminus \{p\}} \GB_{\mathrm{II}} \paren{a_{v}, \beta - a_{v_r}, a_{v_r}; q_{v_r,v}},
	\end{align}
    where $v_r=\mathfrak{pa}^{(p)}(v)$ for $v\in V\setminus \{r\}$. Then Remark \ref{rem1} yields
	\begin{align}
		K \sim \TB^{G}_{\mathrm{II}}(a, \beta, q),
	\end{align}
	which completes the induction step and the proof. 
\end{proof}

\bigskip\begin{proof}[Proof of Theorem \ref{Charact} (i)]
Set	$X^{(r)} \coloneqq \Psi^{(r)}(K)$, $r \in V$. Then \eqref{well_def} implies  $\Phi^{(r)}(X^{(r)}) = K$, $r\in V$. The proof essentially follows the lines of the proof of (ii). It proceeds by induction on $|E|$. 

In case $|E|=1$, we rely on Lemma \ref{rootchange}, the bivariate characterization given in Proposition \ref{2bI} and Theorem \ref{direct} - repeating the argument of the previous case.

In case $|E| > 1$, as in the proof of (ii), we choose $\ell\in L_G$, $p\in V$, denote $\check V = V \setminus \{\ell\}$ and $\check G = G_{\check V}$. 
    Let $\check{K}$ be defined by
	\begin{align}
		\check K_{p} &:= K_{p} \frac{1 - q_{\ell, p} K_{\ell}}{1 - K_{\ell}},
		&
		\check K_{u} &:= K_{u}, \quad u \in \check{V} \setminus \{p\}.
	\end{align}
    Under the assumption that 
    \begin{equation}\label{ass}
        \check K\in \mathcal D_{\check G}\quad \text{a.s.}
    \end{equation}
    the proof proceeds, step by step,  as in case (ii) with obvious modifications: $\check{\Psi}^{(r)}$ in place of $\check{\psi}^{(r)}$, $H_{q_{p,\ell}}$ in place of $h_{q_{p,\ell}}$ and Proposition \ref{2bI} in place of Proposition \ref{2bII}. So we obtain $K\sim \mathrm{TB}_{\mathrm I}^G(a,\beta,q)$.

Hence, the proof will be completed when we prove \eqref{ass}. To this end, we choose $r\in L_G\cap L_{\check G}$ and consider a function $\Theta:\mathcal D_G\to\mathbb R^{\check V}$ defined by
\[
\Theta(k)=\Phi_{\check G}^{(r)}((\Psi^{(r)}_{G,v}(k))_{v\in \check V}),\quad k\in\mathcal D_G.
\]
Note that Lemma \ref{well_def} implies that $\Theta(\mathcal D_G)\subseteq\mathcal D_{\check G}$. Referring again to Lemma \ref{well_def} and to recursion of Lemma \ref{pp_rep}, we get
\begin{align}
\Theta_v(k)=&\Phi_{\check G,v}^{(r)}((\Psi^{(r)}_{G,v}(k))_{v\in \check V})=\Psi^{(r)}_{G,v}(k)\,\prod_{w\in\mathfrak{ch}_{\check G}^{(r)}(v)}\,\tfrac{1-\Psi^{(r)}_{G,w}(k)}{1-q_{w,v}\Psi^{(r)}_{G,w}(k)}\nonumber\\
=&k_v\Biggl(\prod_{u\in\mathfrak{ch}_G^{(r)}(v)}\,\tfrac{1-q_{u,v}\Psi^{(r)}_{G,u}(k)}{1-\Psi^{(r)}_{G,u}(k)}\Biggr)\,\prod_{w\in\mathfrak{ch}_{\check G}^{(r)}(v)}\,\tfrac{1-\Psi^{(r)}_{G,w}(k)}{1-q_{w,v}\Psi^{(r)}_{G,w}(k)}.\label{theta}
\end{align}
Note that for $v\in\check V\setminus\{p\}$ we have $\mathfrak{ch}_{G}^{(r)}(v)=\mathfrak{ch}_{\check G}^{(r)}(v)$ and $\mathfrak{ch}_{G}^{(r)}(p)=\mathfrak{ch}_{\check G}^{(r)}(p)\cup\{\ell\}$. Therefore, \eqref{theta} gives $\Theta_v(k)=k_v$ for $v\in \check V\setminus\{p\}$ and $\Theta_p(k)=k_{p} \tfrac{1 - q_{\ell, p} k_{\ell}}{1 - k_{\ell}}$. Consequently, $\Theta(K)=\check K\in\mathcal D_{\check G}$.
\end{proof}
\section{Tree-Kummer model as a scaling limit of tree-$\mathrm{Beta}_{\mathrm{II}}$ models} \label{sec:limits}

This section is devoted to establishing scaling relations for $\delta_{G}$ that facilitate the passage from tree-beta laws of the second kind to tree-Kummer laws.
\begin{definition}
	\label{def:polynomials-I-II-III}
	Let $G = (V, E)$ be an undirected tree with edge weights $q = (q_{e})_{e \in E} \in \R^{E}$.
	For $k = (k_{v})_{v \in V} \in \R^{V}$ define
	\begin{align}\label{DKum}
		\Delta_{G}^{\mathrm{Kum}}(k|q) \coloneqq \sum_{\substack{\varnothing \ne U \subseteq V \\G_U \, \text{is a tree}}} q^{E_U}k^{U}.
	\end{align}
\end{definition}
The density of the tree-Kummer distribution $\mathrm{TK}^G(a,\beta,q)$, see \eqref{tKum}, can be written as
\[
f(k)\propto e^{-\beta\Delta_{G}^{\mathrm{Kum}}(k|q)}\;\Biggl(\prod_{v\in V}\,k_v^{a_v-1}\Biggr)\,\mathbf{1}_{(0,\infty)^V}(k),
\]
with $e^{-\beta}$ absorbed into the normalizing constant. It appears that this law is a limiting case of the second kind tree-beta distribution. 

\begin{theorem}
	\label{prop:tree-kummer-limit}
	Let $G = (V, E)$ be a  tree with egde weights  $q\in(0,\infty)^E$.
	Fix $a = (a_{v})_{v \in V} \in(0, \infty)^{V}$  and $\beta > 0$.
	For any $\varepsilon>0$ such that $\varepsilon a\in(0,\beta)^V$ let  $K^{(\varepsilon)}$ be a random vector with the second kind  tree beta distribution,
	\begin{align}
		K^{(\varepsilon)} \sim \TB^{G}_{\mathrm{II}} \paren{a, \beta / \varepsilon, q/\varepsilon}.
	\end{align}
	
	Then, as $\varepsilon \to 0$, 
    \begin{equation}\label{zbiez}
    \varepsilon^{-1}\,K^{(\varepsilon)}\Rightarrow K\sim\mathrm{TK}^G(a, \beta, q) \quad \text{in distribution.}
	\end{equation}
\end{theorem}

The proof is based on scaling limit properties of the polynomial $\delta_G(x)=\delta_G(x|q)$.
\begin{lemma}
	\label{lim1}
	Let $G = (V, E)$ be a tree and let $q \in \R^{E}$.
	For each fixed $k \in \R^{V}$,
	\begin{align}
		\lim_{\varepsilon \downarrow 0}\,\delta_G(\varepsilon k\,|q/\varepsilon)^{1 / \varepsilon}
		=
		\exp \{\Delta_G^{\mathrm{Kum}}(k\,|\,q) \}.
	\end{align}
\end{lemma}
\begin{proof}
Since $G=(V,E)$ is a tree, for every $U \subseteq V$ one has $|U| - |E_U| = \pi(U)$, where $\pi(U)$ is the number of connected components of $G_U$. Thus, directly from the definition of $\delta_G$ we get
\begin{align*}
\delta_{G}(\varepsilon k\,|\,q/\varepsilon)=&\sum_{U \subseteq V} \varepsilon^{|U| - |E(U)|}q^{E(U)}k^{U}=\sum_{U \subseteq V} \varepsilon^{\pi(U)}q^{E(U)}k^{U}\\
=&1+\varepsilon\sum_{\substack{\varnothing\neq U \subseteq V\\ G_U\,\text{is a tree}}}\,q^{E(U)}k^{U}+\varepsilon^2\sum_{\substack{U \subseteq V\\ \pi(U)\ge 2}} \varepsilon^{\pi(U)-2}q^{E(U)}k^{U},
\end{align*}
which, in view of \eqref{DKum}, gives 
\begin{align}\label{dapp}
		\delta_{G}(\varepsilon k\,|\,q/\varepsilon)
		=
		1 + \varepsilon \Delta_{G}^{\mathrm{Kum}}(k\,|\,q)
		+ o(\varepsilon). 
	\end{align}
	
    Set
	\begin{align}
		t_{\varepsilon} \coloneqq \frac{\delta_{G}(\varepsilon k\,|\,q/\varepsilon) - 1}{\varepsilon}=\Delta_{G}^{\mathrm{Kum}}(k\,|\,q)+\tfrac{o(\varepsilon)}{\varepsilon},
	\end{align}
	so that $\delta_{G}(\varepsilon k\,|\,q/\varepsilon) = 1 + \varepsilon t_{\varepsilon}$ and $t_{\varepsilon} \to \Delta_{G}^{\mathrm{Kum}}(k\,|\,q)$ as $\varepsilon \downarrow 0$.
	The claim follows from the elementary limit $(1 + \varepsilon t_{\varepsilon})^{1 / \varepsilon} \to e^{t}$ whenever $t_{\varepsilon} \to t$.
\end{proof}

\bigskip
\begin{proof}[Proof of Theorem \ref{prop:tree-kummer-limit}]
The density of $\varepsilon^{-1}\,K^{(\varepsilon)}$, given in \eqref{denII}, takes the form
	\begin{align}
		\underbrace{c_{\mathrm{II}}(\varepsilon)\,\varepsilon^{S(a)}}_{J(\varepsilon)}\,\times\,\underbrace{\delta_{G}(\varepsilon k\,|\,q/\varepsilon)^{- \beta / \varepsilon}
		\Bigl(\prod_{v \in V}k_{v}^{a_{v} - 1}\Bigr) }_{M(\varepsilon,k)}\,\times \,\1{(0, \infty)^{V}}(k),
	\end{align}
    where $S(a)=\sum_{v\in V}\,a_v$. In view of Lemma \ref{lim1},
    \begin{equation}\label{Mk}
    \lim_{\varepsilon\downarrow 0}\,M(\varepsilon,k)=\exp \{- \beta \Delta_G^{\mathrm{Kum}}(k\,|\,q) \}\,\prod_{v \in V}k_{v}^{a_{v} - 1}.
    \end{equation}
    
Note that $\delta_{G}(\varepsilon k\,|\,q/\varepsilon)\ge 1+\varepsilon \sum_{v\in V}\,k_v$. Take $c>0$ (to be specified later). Since the function $\varepsilon\mapsto (1+\varepsilon t)^{1/\varepsilon}$, $\varepsilon\in(0,c)$, is decreasing, we have
    \[
    \delta_{G}(\varepsilon k\,|\,q/\varepsilon)^{-\beta/\varepsilon}\le (1+c\sum_{v\in V}k_v)^{-\beta/c}.
    \]
    which yields 
    \[
    M(\varepsilon,k)\le (1+c\sum_{v\in V}k_v)^{-\beta/c}\,\prod_{v \in V}k_{v}^{a_{v} - 1}.
    \]
    Taking $c$ such that $S(a)<\tfrac{\beta}c$ we conclude that the right-hand side above is integrable over $(0,\infty)^V$. Thus by the Lebesgue dominated convergence, in view of \eqref{Mk}, we have
    \begin{equation}\label{jeden}
    1=\lim_{\varepsilon\downarrow 0}\,J_{\varepsilon} \int_{(0,\infty)^V}\,M(\varepsilon,k)\,\mathrm dk=\left(\lim_{\varepsilon\downarrow 0}\,J_{\varepsilon}\right)\,\int_{(0,\infty)^V}\,\exp \{- \beta \Delta_G^{\mathrm{Kum}}(k\,|\,q) \}\,\prod_{v \in V}k_{v}^{a_{v} - 1}\,\mathrm dk.
    \end{equation}
    Consequently, the limit $\lim_{\varepsilon\downarrow 0}\,J_{\varepsilon}$ exists and, due to \eqref{jeden}, it is exactly the normalizing constant of the tree-Kummer density \eqref{tKum}. Thus the proof is complete.
    \end{proof}

From the above proof, one can relatively easily decipher the normalizing constant $c_K$ of the $\mathrm{TK}^G(a,\beta,q)$ distribution, which seems to be not available in the literature. 
\begin{proposition}\label{ck}
The normalising constant $c_K$ of the tree-Kummer distribution  $\mathrm{TK}^G(a, \beta, q)$, $a\in(0,\infty)^V$, $\beta>0$, $q\in(0,\infty)^V$ has the form
\begin{equation}\label{ckform}
c_K=\biggl(\prod_{v\in V}\,\tfrac{\beta^{a_v(1-d_v/2)}}{\Gamma(a_v)}\biggr)\,\prod_{\{w,v\}\in E}\,\sqrt{\tfrac{q_{w,v}^{a_u+a_v}}{U\left(a_w,a_w-a_v+1,\beta/q_{w,v}\right)\,U\left(a_v,a_v-a_w+1,\beta/q_{w,v}\right)}},
\end{equation}
where $d_v$ denotes the degree of vertex $v\in V$ and $U$ is the Tricomi function.
\end{proposition}

The proof below is based on limit connection between  ${}_2F_1$  and the Tricomi function. Another possibility would be to rely directly on IP property of the tree-Kummer law, see \cite{PilWes16}[Theorem 2].

Recall that the Tricomi's \(U\)-function has the integral representation (see, e.g. 13.4.4 in \cite{NIST2025})
\[
U(\alpha,\beta,z)=\tfrac{1}{\Gamma(\alpha)}\,\int_0^\infty\,e^{-zt}t^{\alpha-1}(1+t)^{\beta-\alpha-1}\,\mathrm dt.
\]

\begin{proof}[Proof of Proposition \ref{ck}]
Recall that $c_K=\lim_{\varepsilon \downarrow 0}\,J(\varepsilon)$ with  $J(\varepsilon):=c_{\mathrm{II}}(\varepsilon)\,\varepsilon^{S(a)}$.  From  \eqref{cII} it follows that
    \begin{equation}\label{jeps}
J(\varepsilon)=\Biggl(\prod_{v\in V}\,\bigl[\tfrac1{\Gamma(a_v)}\,\tfrac{\varepsilon^{a_v}\,\Gamma(\beta/\varepsilon)}{\Gamma(\beta/\varepsilon-a_v)}\bigr]\Biggr)\,
    \,\prod_{\{w,u\}\in E}\,\tfrac 1{{}_{2}F_{1}\bigl(a_{w}, a_{u}, \beta/\varepsilon; 1 - q_{w, u}/\varepsilon\bigr)}
    \end{equation}
     We recall (see e.g. formula 5.11.12 in \cite{NIST2025})
    \begin{equation*}\label{NIST1}
    \tfrac{\Gamma(z+A)}{\Gamma(z+B)}\sim z^{A-B},\quad \text{as }\;z\to\infty,\quad A,B\in \mathbb{R}.
    \end{equation*}
    which for $z=\tfrac{\beta}{\varepsilon}$, $A=0$ and $B=-a_v$ gives
    \begin{equation}\label{Gamma}
    \lim_{\varepsilon\downarrow 0}\tfrac{\varepsilon^{a_v}\,\Gamma(\beta/\varepsilon)}{\Gamma(\beta/\varepsilon-a_v)}=\beta^{a_v},\quad v\in V.
    \end{equation}
 We also recall (see  \cite{EMOT1953}[Sec. 6.8(1)] that 
 \[
 \lim_{\gamma \to \infty}\,{}_2F_1\!\left(\alpha,\beta;\gamma;1-\tfrac{\gamma}{z}\right)=z^\alpha U\!\left(\alpha,\alpha-\beta+1,z\right),
 \]
 which for $a_u=\alpha$, $a_v=\beta$, $\gamma=\beta/\varepsilon$, $z=\beta/q_{w,u}$, on noting that ${}_2F_1$ function is symmetric in the first two parameters, gives
 \begin{equation}\label{Trico}
 \lim_{\varepsilon\downarrow 0}\,{}_2F_1\!\left(a_u,a_w;,\tfrac{\beta}\varepsilon;1-\tfrac{q_{w,u}}{\varepsilon}\right)=\left(\tfrac{\beta}{q_{w,u}}\right)^{a_u}\, U\!\left(a_u,a_u-a_w+1,\beta/q_{u,w}\right).
 \end{equation}
Thus, using \eqref{Gamma} and \eqref{Trico} in \eqref{jeps}, we get  
\begin{align}\label{cck}
c_K=\lim_{\varepsilon \downarrow 0}\,c_{\mathrm{II}}(\varepsilon)\,\varepsilon^{S(a)}=\Biggl(\prod_{v\in V}\,\tfrac{\beta^{a_v}}{\Gamma(a_v)}\Biggr)\,\prod_{\{w,u\}\in E}\,\left(\tfrac{q_{w,u}}{\beta}\right)^{a_u}\,\tfrac1{U(a_u,a_u-a_w+1,\beta/q_{w,u})}.
\end{align}
The second product above does not look symmetric in $u$ and $v$, but it actually, is: 13.2.40 of \cite{NIST2025} says that
\[
U(A,B,z)=z^{1-B}U(A-B+1,2-B,z).
\]
For $A=a_u$, $B=a_u-a_w+1$ and $z=\beta/q_{w,u}$ it gives 
\[
\left(\tfrac\beta{q_{w,u}}\right)^{a_w}\,U\Bigl(a_w,a_w-a_u +1,\tfrac\beta{q_{w,u}}\Bigr)=\left(\tfrac\beta{q_{w,u}}\right)^{a_u}\,U\left(a_u,a_u-a_w +1,\tfrac\beta{q_{w,u}}\right).
\] 
This identity applied to \eqref{cck}, on noting that $\sum_{\{u,w\}\in E}\,(a_u+a_w)=\sum_{v\in V}\,d_va_v$ (recall that $d_v$ is the degree of the vertex $v\in V$), gives \eqref{ckform}.
\end{proof}

\bibliographystyle{amsplain}
\bibliography{tree_beta_references}
\end{document}